\documentclass[a4paper,12pt]{article}

\usepackage{amsfonts,amssymb,amsbsy,amsmath,amsthm,mathrsfs,enumerate,verbatim}

  \usepackage{graphics} 
  \usepackage{epsfig} 
\usepackage{graphicx}  \usepackage{epstopdf}
 \usepackage[colorlinks=true]{hyperref}

\hypersetup{urlcolor=blue, citecolor=red}

\usepackage{amsfonts}
\usepackage{amsmath}
\usepackage{amssymb}
\usepackage{graphicx}
\usepackage{epsf}
\usepackage{epsfig}

\usepackage{color}
\usepackage{afterpage}
\usepackage{verbatim}

\newcommand{\norm}[1]{\left\Vert#1\right\Vert}
\newcommand{\abs}[1]{\left\vert#1\right\vert}

\usepackage{times}

\newtheorem{theo}{Theorem}

\numberwithin{equation}{section}

\newtheorem{theorem}{Theorem}[section]
\newtheorem{proposition}[theorem]{Proposition}

\theoremstyle{definition}
\newtheorem{remark}[theorem]{Remark}

\newcommand{\bel}{\begin{equation} \label}
\newcommand{\ee}{\end{equation}}

\def\beq{\begin{equation}}
\def\eeq{\end{equation}}
\newcommand{\bea}{\begin{eqnarray}}
\newcommand{\eea}{\end{eqnarray}}
\newcommand{\beas}{\begin{eqnarray*}}
\newcommand{\eeas}{\end{eqnarray*}}

{

 \usepackage{graphicx,color}
 \definecolor{mygreen}{cmyk}{1,0,1,0.1}

\usepackage{epsf}
\usepackage{epstopdf}

\usepackage{pdfsync}

\graphicspath{
  {FIGURES/}
}

\begin{document}

\title{On stabilized $P1$ finite element approximation  for
time harmonic   Maxwell's equations}


\author{M. Asadzadeh$^*$ and L. Beilina 
\thanks{ Department of Mathematical Sciences, Chalmers University of Technology and
University of Gothenburg, SE-42196 Gothenburg, Sweden, e-mail:   
 e-mail:{mohammad@chalmers.se}, e-mail:{larisa@chalmers.se}}}


\date{}


\maketitle

\begin{abstract}

One way of improving the behavior of finite element schemes for 
classical, time-dependent Maxwell's equations, is to render 
them from their hyperbolic character to elliptic form. 
  This paper is devoted to the study of the stabilized  
linear finite element method for the time harmonic Maxwell's equations 
in a dual form obtained through the Laplace transformation in time.  
The model problem is for the particular case of the dielectric permittivity  
function which is assumed to be constant in a boundary neighborhood. For 
the 
stabilized model a coercivity relation is derived that 
guarantee's the existence of a 
unique solution for the 
discrete problem. The convergence is addressed both in {\sl a priori} 
and {\sl a posteriori} settings. 
In the a priori error estimates   
we confirm the theoretical convergence of the scheme in 
a $L_2$-based, gradient dependent, triple norm. The order of convergence is 
 ${\mathcal O}(h)$ in weighted Sobolev space $H^2_w(\Omega)$, and hence 
optimal. Here, $w=w(\varepsilon ,s)$ where $\varepsilon$ is the 
dielectric permittivity function and $s$ is the Laplace transformations 
variable. 
We also derive, similar, optimal { a posteriori} error estimates 
 controlled by 
a certain, weighted, norm of the residual of the computed solution.  
The  posteriori approach is used for constructing adaptive 
algorithms  
 for the computational purposes. Further, assuming a sufficiently regular 
solution for the dual problem, we reach the same convergence of 
${\mathcal O}(h)$. Finally, through implementing  
several numerical examples, we validate the robustness of the proposed scheme. 

\end{abstract}

\section{Introduction}

By the growing efficiency of the recent 
 computing facilities, 
the development of efficient computational methods for simulation of 
partial differential equations in two and three-dimensions, in particular 
when
the computational domains are very large, which are of  vital interest, 
has become achievable. 
It is very typical that in industrial applications 
computational domains often comprise very large subdomains 
with constant values of material parameters. Usually, 
only some part of these domains where there is an over-admissible 
 material change presents extra caution and therefore is of 
interest. In such cases the problem is described by a model
equation with the constant material parameters in a boundary
neighborhood of the computational domain. Some applications of the    
Maxwell's models, describing the electro-magnetic fields,  
fit into this category. 

It is well known that for the stable implementations of the finite element
solution for the Maxwell's equations, divergence-free edge elements are the
most satisfactory choice from a theoretical point of view 
~\cite{Monk, Nedelec}.  
However, the edge elements are less attractive for solution
of time-dependent problems since a linear system of equations should
be solved at every time iteration, a procedure requiring an unrealistic 
degree of time resolution, and hence  expensive from the
 programming point of view.

In contrary, continuous P1 finite elements provide both reliable and 
efficient (inexpensive)  way for numerical simulations, in particular
 compared to $H(curl)$
conforming methods. They can be efficiently used in a fully explicit
finite element scheme with lumped mass matrix \cite{delta, joly}.
However, P1 elements applied to the solution of Maxwell's equations
have a number of drawbacks, e.g., in problems with re-entered corners
and non-zero tangential components at the boundary, they result in a
spurious oscillatory solutions \cite{MP, PL}.  
There are a number of techniques
which remove such oscillatory solutions from the numerical approximations of
Maxwell's equations, see for example 
\cite{Jiang1, Jiang2, Jin,  div_cor, PL}.

We circumvent these difficulties considering convex
computational domains with constant
values of parameters in a boundary neighborhood.
More specifically,
in this work we  consider stabilized P1 finite element method 
for the numerical solution
of time harmonic Maxwell's equations for the special case when the
dielectric permittivity function has a constant value in a boundary
neighborhood.  
In this way the Maxwell's equations are 
transformed to a set
of time-independent wave equations on the boundary neighborhood. 
Thus, several type of adequate boundary conditions for these
equations might be handled by P1 finite
elements.

Recently, stability and
consistency of the stabilized P1 finite element method
 for time-dependent Maxwell's equations was presented in 
\cite{BR}. 
Efficiency in usage of an explicit P1 finite element scheme is evident
for solution of {\sl Coefficient Inverse Problems} (CIPs).  
In many algorithms which solve electromagnetic
CIPs a qualitative collection of experimental data (measurements) is
necessary at the boundary of the computational domain to determine the
dielectric permittivity function inside it.  In this case the
numerical solution of time-dependent Maxwell's equations are required
in the entire space $\mathbb{R}^{3}$, see for example \cite{BK, BTKM1,
  BTKM2, BondestaB, TBKF1, TBKF2}, and it is efficient to consider
Maxwell's equations with constant dielectric permittivity function in
a neighborhood of the boundary of the computational domain.  An
explicit P1 finite element scheme in non-conductive media
is numerically tested for solution of time-dependent Maxwell's system both 
in 2D and 3D cases in \cite{BMaxwell}. 
The  P1 finite element scheme of \cite{BMaxwell} is used
for solution of different CIPs to determine the dielectric
permittivity function in non-conductive media for time-dependent
Maxwell's equations using simulated and experimentally generated data,
see, e.g., 
\cite{BMedical, BMaxwell, BK,  BTKM1, BTKM2, BondestaB, TBKF1, TBKF2}. 
In this study we derive optimal a priori and a posteriori 
convergence rates for the P1 finite element 
scheme, for the time harmonic Maxwell system, assuming a constant  dielectric 
permittivity function at a boundary neighborhood, 
hence with no in- and outflow.  
 
An outline of this paper is as follows. 
In Section 2 we introduce the mathematical model and present the Cauchy problem 
for the time harmonic Maxwell's equations, where we assumed 
no dielectric volume 
charge. In Section 3 we describe variational method for the stabilized model, 
set up the finite element scheme and prove its well-posedness. 
Section 4 is devoted to the error analysis, where optimal a priori and a 
posteriori error estimates are derived in a, gradient dependent, triple norm 
of Sobolev type. In the a posteriori case the boundary residual, containing a 
normal derivative, is balanced by a multiplicative positive power of the 
mesh parameter. Further regularity assumptions on the solution of 
a dual problem yields a superconvergence for the a posteriori error estimate 
of order ${\mathcal O}(h^{3/2}).$  An adaptivity algorithm is presented 
in order to reach an adequate stopping criterion in iterative mesh-generation 
procedure. Section 5 is devoted to implementations that justify the 
theoretical/approximative achievements in the paper. Finally, in Section 6 we 
conclude the results of the paper. 

Throughout this paper $C$ denotes a generic constant, 
not necessarily the same at each occurrence and independent of 
the mesh parameter, the solution and other involved parameters, 
unless otherwise specifically specified. 

\section{The mathematical model}

The original model here is given in terms of the electric field
 $\hat{E}\left( x,s\right), x \in
\mathbb{R}^{d}, d=2,3$ and is varying with  the   pseudo-frequency 
$s > const. > 0$, see \cite{BK}, under the 
assumption  that the magnetic permeability of 
the medium is $\mu \equiv 1$.  
We consider the Cauchy problem for the time-harmonic Maxwell's equations for
  electric field $\hat{E}\left( x,s\right)$  under the assumption of
 the vanishing  electric  volume
  charges. Hence, 
  the corresponding equation is then given by 
\begin{equation}\label{model}
\begin{split}
 s^2  \varepsilon(x) \hat{E}(x,s)   +  \nabla \times \nabla \times \hat{E}(x,s)  &= s \varepsilon(x)  f_0(x),~~ x \in \mathbb{R}^{d}, d=2,3 \\
  \nabla \cdot(\varepsilon(x) \hat{E}(x,s)) &= 0. 
\end{split}
\end{equation}
Here, $ \varepsilon(x) = \varepsilon_r(x) \varepsilon_0$ is the
 dielectric permittivity function,  $\varepsilon_r(x)$ is the
dimensionless relative dielectric permittivity function
and $\varepsilon_0$ is the permittivity of
the free space, and 
\begin{equation}\label{divfree} 
\nabla \times \nabla \times {E}= \nabla(  \nabla\cdot E)-\nabla^2E. 
\end{equation}
\noindent 
The model equation \eqref{model}  can be  obtained by applying the 
Laplace transform in time, viz. 
\begin{equation} \label{laplace}
\begin{split}
\widehat{E}(x, s) &:=  \int_{0}^{+\infty}  E(x,t) e^{-s t} dt, \qquad s =const. >0 
\end{split}
\end{equation}
 to the function
$E\left( x,t\right)$ satisfying  the
time-dependent Maxwell's equations
\begin{equation}\label{E_gauge}
\begin{split}
  \varepsilon(x) \frac{\partial^2 E(x,t)}{\partial t^2} +  
\nabla \times \nabla \times E(x,t)  &= 0, 
~~ x \in \mathbb{R}^{d}, d=2,3,\,\, t\in (0,T].\\
   \nabla \cdot(\varepsilon E)(x,t) &= 0, \\
  E(x,0) = f_0(x), ~~~\frac{\partial E}{\partial t}(x,0) &= 0,~~ x \in \mathbb{R}^{d}, \quad d=2,3.
\end{split}
\end{equation}
In the problem \eqref{E_gauge} we consider non-zero initial condition
$E(x,0) = f_0(x)$ which is important for solution of coefficient
inverse problem for determination of the function $\varepsilon(x)$
from finite number of observations at the boundary. This condition
gives stability estimate and uniqueness of the reconstruction of the
function $\varepsilon(x)$ using finite number of observations of the
electric field $E(x,t)$ at the small neighborhood of the boundary
$\Gamma$, see details in \cite{BCS}.

We are not able, numerically, solve the problem (\ref{E_gauge}) in 
unbounded domains and thus introduce a convex bounded, polygonal, subdomain
$\Omega\subset \mathbb{R}^{d}, d=2,3$ with  boundary $\Gamma$: 
More specifically, $\Omega$ is a simply connected domain. We define   
$\Omega_2 :=\Omega \setminus \Omega_1$, where $\Omega_1\subset \Omega$ 
has positive Lebesgue measure and 
$\partial\Omega\cap\partial\Omega_1=\emptyset$. In this way 
cutting out $\Omega_1$ from $\Omega$, 
the new subdomain $\Omega_2$ shares the boundary with both 
$\Omega$ and $\Omega_1$:  
$\partial\Omega_2=\partial\Omega\cup\partial\Omega_1$, 
$\Omega=\Omega_{\rm 1}\cup\Omega_2$,                              
$\Omega_1 =\Omega \setminus \Omega_2$ and 
$\bar\Omega_{\rm 1} \cap \bar\Omega_{\rm 2}=\partial \Omega_{\rm 1} $, 
(see the Fig. \ref{fig:F0}).

\begin{figure}[tbp]
\begin{center}
\begin{tabular}{cc}
  {\includegraphics[scale=0.25,clip=]{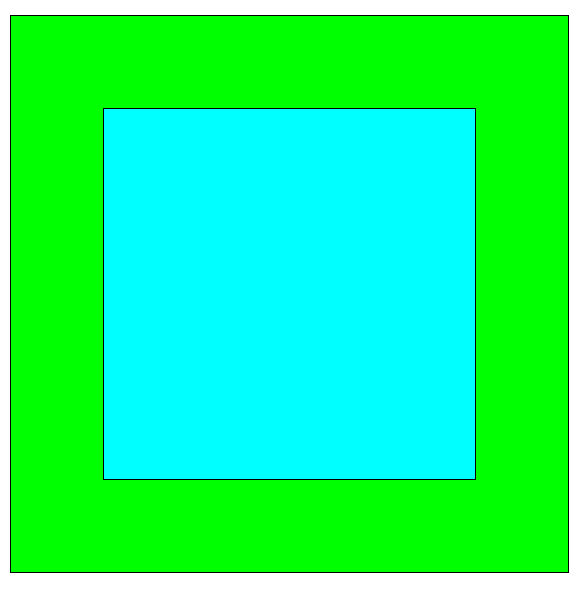}}
\put(-70,60){$\Omega_{1}$}
\put(-70,10){$\Omega_{2}$}
\put(-110,10){$\varepsilon=1$}
\put(-110,90){$\varepsilon \in [1,d_1]$}
&
 {\includegraphics[scale=0.25, clip=]{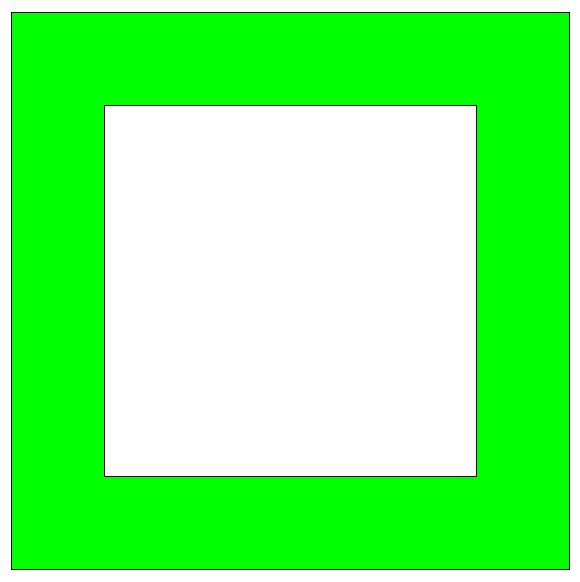}}
 \put(-70,10){$\Omega_{2}$}
\put(-110,10){$\varepsilon=1$}
\\
 a) $ \Omega = \Omega_{1} \cup \Omega_{2}$ &
 b) $\Omega_{2}$ 
\end{tabular}
\end{center}
\caption{\small \emph{\   Domain   decomposition in $\Omega$.
}}
\label{fig:F0}
\end{figure}
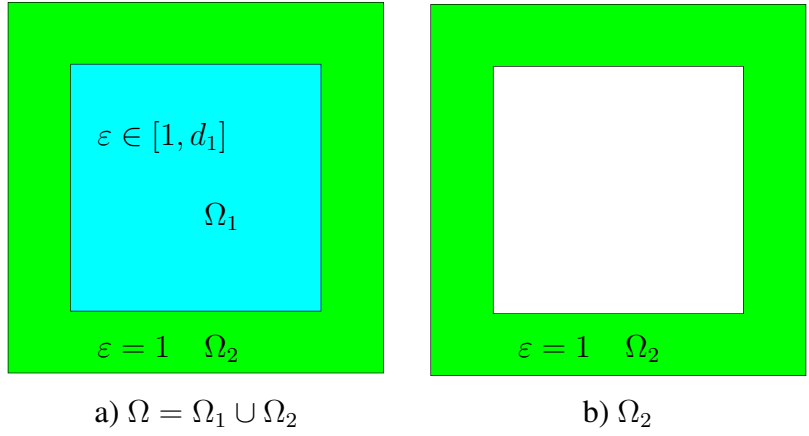

To proceed we  assume that $\varepsilon(x) \in C^{2}(\mathbb{R}^{d}), 
d=2,3$ and 
 for some known constant $d_1>1$, 
\begin{equation} \label{2.3}
\begin{split}
  \varepsilon(x) \in \left[ 1,d_1\right], \quad ~\text{ for }\,\,\,& x\in  \Omega _{1}, \\
  ~  ~ \varepsilon(x) =1, \qquad \text{ for } \,\,\,& x\in  \Omega \setminus \Omega _{1} , \\
\partial_\nu\varepsilon =0 , \qquad 
\text{ for }\,\,\, & x\in \partial\Omega_2. 
\end{split}
\end{equation}
\begin{remark}\label{Remark1}
Conditions \eqref{2.3} mean that, 
in the vicinity of the boundary of the computational domain $\Omega$, 
 the equation \eqref{E_gauge}
transforms to the usual time-dependent wave equation.
\end{remark} 
As for $\Gamma:=\partial\Omega$ the boundary of the computational domain
$\Omega$, we use the split $\Gamma =\Gamma_{1} \cup \Gamma_{2} \cup
\Gamma_{3}$, so that $\Gamma_{1}$ and $\Gamma_{2}$ are the top
and bottom sides, with respect to $y$- (in $2d$) or $z$-axis (in $3d$),  
of the domain $\Omega$, respectively, 
while $\Gamma_{3}$ is the
rest of the boundary. Further, $\partial_\nu(\cdot)$ denotes 
the normal derivative on $\Gamma$,  where $\nu$ is the 
outward unit normal vector on the boundary $\Gamma$.

\begin{remark}\label{Remark2}
In most estimates below, it suffices to 
restrict the Neumann boundary condition 
for the dielectric permittivity function to: $\partial_\nu\varepsilon (x)=0$,  
 to $\Gamma_1\cup \Gamma_2$. Possible usage of the general 
form in \eqref{2.3} 
will 
be clear from the context.              

\end{remark}

Now, using similar argument as in the  studies in 
\cite{BMedical}-\cite{BTKM2} and by Remark \ref{Remark1},  
for the time-dependent wave equation, we
impose first order absorbing boundary condition \cite{EM} 
at $\Gamma_1\cup \Gamma_2$: 
\begin{equation}\label{absorb}
\partial _{\nu }E+\partial _{t}E=0,\qquad 
\left( x,t\right) \in (\Gamma_1 \cup \Gamma_2) \times ( 0,T] . 
\end{equation}
To impose boundary conditions at $\Gamma_3$ we can assume that the
surface $\Gamma_3$ is located far from the domain $\Omega_1$. Hence,
we can  assume that $E\approx E^{inc}$ in a vicinity of 
$\Gamma_3$, where  $E^{inc}$ is the incident field. 
Thus, at $\Gamma_3$ we may impose Neumann boundary condition
\begin{equation} \label{neumann}
\partial _{\nu }E=0,\qquad \left( x,t\right) \in \Gamma_3 \times (0,T].  
\end{equation}

\noindent 
Finally, using  the well
known vector-analysis relation \eqref{divfree}
 and applying the Laplace transform  
to the equation \eqref{E_gauge} and the 
boundary conditions \eqref{absorb}-\eqref{neumann} in the time
 domain, the problem (\ref{model}) will be transformed to the
 following model problem
\begin{equation}\label{model2}
\begin{split}
  s^2 \varepsilon(x) \hat{E}(x,s)   +
\nabla (\nabla
\cdot \hat{E}(x,s)) - \triangle \hat{E}(x,s)
  &= s \varepsilon(x)  f_0(x),~~ x \in \mathbb{R}^{d}, d=2,3 \\
 \nabla \cdot (\varepsilon(x) \hat{E}(x,s)) &= 0, \\
\partial _{\nu } \hat{E}(x,s) &=0, \quad  x \in \Gamma_3, \\
\partial _{\nu } \hat{E}(x,s)  &= f_0(x) - s  \hat{E}(x,s),  
x \in \Gamma_1 \cup \Gamma_2. 
\end{split}
\end{equation}

\section{Variational approach} \label{sec:varform}

We denote the standard inner product in $[L_2(\Omega)]^d$ as
$(\cdot,\cdot), ~ d \in \{2,3\}$, and the corresponding norm by
$\parallel \cdot \parallel$. Similarly we denote by
$\langle \cdot, \cdot\rangle_{\Gamma}$ the standard inner product
of $[L_2(\Gamma)]^{d-1}$ and the associated $L_2(\Gamma)$-norm 
by $\| \cdot \|_{\Gamma}$.  
We define the $L_2$-weighted scalar product
by 
$$
(u,v):= \int_{\Omega} u \cdot v~ d{\bf x},\quad  
(u,v)_{\omega}:= \int_{\Omega} u \cdot v~\omega d{\bf x}, \quad 
\langle u, v \rangle_\Gamma:=\int_\Gamma  u \cdot v~ d{\sigma}, 
$$
and  the $\omega$-weighted $L^2(\Omega)$ norm as 
$$
\| u \|_{\omega}:=\sqrt{\int_{\Omega}  | u|^2 \,\omega d{\bf x}},\qquad 
 ~\omega > 0, \quad \omega \in L^{\infty}(\Omega).
$$

\subsection{Stabilized model}

The stabilized formulation of the problem \eqref{model2}, with 
$ d=2,3$, is now written as:
\begin{equation}\label{model3}
\begin{split}
  s^2  \varepsilon(x) \hat{E}(x,s)   - \triangle \hat{E}(x,s)
  &-  \nabla ( \nabla \cdot (( \varepsilon  -1) \hat{E}(x,s)) =
  s \varepsilon(x)  f_0(x) \quad x \in \mathbb{R}^{d},\\
&\partial _{\nu } \hat{E}(x,s) =0, \quad x \in \Gamma_3, \\
&\partial _{\nu } \hat{E}(x,s)  = f_0(x) - s  \hat{E}(x,s),  \quad x \in \Gamma_1 \cup \Gamma_2, 
\end{split}
\end{equation}
where the second equation in  \eqref{model2} is hidden in the first one above. 
Let us first show that
the  stabilized problem \eqref{model3} 
is equivalent to the original one \eqref{model2}.
To do this we consider the variational formulation of 
\eqref{model3} for all 
 ${\textbf v} \in [H^1(\Omega)]^3$, 
\begin{equation}\label{eq2}
  \begin{split}
 & (s^2 \varepsilon \hat{E}, {\textbf v})
    + (\nabla  \hat{E}, \nabla {\textbf v})  
    +
    (\nabla \cdot (\varepsilon \hat{E}), \nabla \cdot {\textbf v})
    - (\nabla \cdot \hat{E}, \nabla \cdot {\textbf v}) \\ 
  &-\langle f_0, {\textbf v}\rangle_{\Gamma_1\cup\Gamma_2} 
+ \langle s \hat{E}, {\textbf v}\rangle_{\Gamma_1\cup\Gamma_2}
 - \langle\nabla \cdot (\varepsilon \hat{E})  - \nabla \cdot \hat{E}, 
{\textbf v} \cdot  \nu \rangle_{\Gamma}
  = (s \varepsilon f_0,  {\textbf v}). 
  \end{split}
\end{equation}
{\bf Assumption 1.}\label{Assumption1}  By condition \eqref{2.3} we may 
assume that 
the dielectric permittivity function $\varepsilon(x)\equiv 1$ on a 
neighborhood 
of $\Gamma$ and hence the last boundary integral above is indeed $\equiv 0$. 
We have kept this term in order to follow the computational steps. A full 
consideration with  varying $\varepsilon(x)$ is of vital importance 
for the time dependent problem and will be  included in the 
subject of a forthcoming study. For a communication on this see, e.g., 
\cite{BR}. 

Integration by parts, in the spatial domain, in the second, third 
and fourth terms 
in the equation \eqref{eq2} yields, for all 
$ {\textbf v} \in [H^1(\Omega)]^3$, that 
\begin{equation}\label{eq3}
\begin{split}
 & (s^2 \varepsilon  \hat{E}, {\textbf v}) +
  (\nabla (\nabla \cdot \hat{E}), {\textbf v}) - (\triangle \hat{E}, {\textbf v})
  - (\nabla (\nabla \cdot (\varepsilon  \hat{E})),{\textbf v})  
  -\langle f_0, {\textbf v}\rangle_{\Gamma_1\cup\Gamma_2} + 
\langle s \hat{E}, {\textbf v}\rangle_{\Gamma_1\cup\Gamma_2} \\
  &- \langle \nabla \cdot (\varepsilon \hat{E})
  - \nabla \cdot \hat{E}, {\textbf v} \cdot  \nu \rangle_{\Gamma}
+ \langle \partial_\nu\hat{E}, {\textbf v}\rangle_{\Gamma} +
 \langle\nabla \cdot (\varepsilon \hat{E}) 
- \nabla \cdot \hat{E}, {\textbf v} \cdot  \nu \rangle_{\Gamma}  =  
( s \varepsilon f_0,  {\textbf v}),
\end{split}
\end{equation}
which can be simplified as 
\begin{equation}\label{eq3new}
\begin{split}
&  (s^2 \varepsilon \hat{E}, {\textbf v}) +
  (\nabla (\nabla \cdot \hat{E}), {\textbf v}) - 
(\triangle \hat{E}, {\textbf v})
  - (\nabla (\nabla \cdot (\varepsilon \hat{E})),{\textbf v})  \\
  &- \langle f_0, {\textbf v}\rangle_{\Gamma_1\cup\Gamma_2} + 
\langle s \hat{E}, {\textbf v}\rangle_{\Gamma_1\cup\Gamma_2} 
+ \langle \partial_\nu \hat{E}, {\textbf v}\rangle_{\Gamma_1\cup\Gamma_2}  = 
 (s \varepsilon f_0,  {\textbf v}),
\qquad \forall {\textbf v} \in [H^1(\Omega)]^3. 
\end{split}
\end{equation}
Using the conditions \eqref{2.3}   we get from the above equation that 
\begin{equation}\label{eq3bis}
  \begin{split}
  s^2  \varepsilon \hat{E} +
  \nabla (\nabla \cdot \hat{E}) - \triangle \hat{E} 
- \nabla (\nabla \cdot (\varepsilon \hat{E})) &=  s \varepsilon f_0\,\,\,   
 \qquad \,\mbox{ in } \Omega,\\
  \partial _{\nu}\hat{E} &= f_0 - s\hat{E}  \quad
\mbox{ on } \Gamma_1 \cup \Gamma_2,\\
      \partial_{\nu} \hat{E} &= 0  \qquad\qquad \mbox{ on } \Gamma_3.
  \end{split}
\end{equation}
From this equation it follows that $\nabla \cdot (\varepsilon  \hat{E})=0$.  
To see it, we let $\tilde{E}$ be the unique solution of the problem
\eqref{model2} and consider the difference 
$\bar{E}=\hat{E} - \tilde{E}$ between the solution $\hat{E}$ of the problem 
\eqref{eq3bis} and the solution $\tilde{E}$  of the problem \eqref{model2}.
We observe that the function $\bar{E}$ is the solution of the following 
boundary value problem:
\begin{equation}\label{eq5}
  \begin{split}
 s^2  \varepsilon \bar{E}  - \Delta \bar{E} 
    - \nabla( \nabla \cdot ((\varepsilon - 1)  \bar{E})  ) &= 0   
\,\,\,\quad \qquad\mbox{ in } \Omega, \\
    \partial_{\nu} \bar{E} &= -s  \bar{E} \qquad \mbox{ on } \Gamma_1 \cup \Gamma_2, \\
      \partial_{\nu} \bar{E} &= 0 \,\,\, \quad\qquad  \mbox{ on } \Gamma_3.
  \end{split}
\end{equation}
Now we multiply the equation (\ref{eq5}) by 
${\textbf v} \in [H^1(\Omega)]^3$ and integrate 
over $\Omega$ to get:
\begin{equation}\label{aux0}
\begin{split}
  (s^2  \varepsilon \bar{E} , {\textbf v}) 
  + (\nabla  \bar{E},\nabla {\textbf v})  + ((\varepsilon  - 1) \nabla \cdot \bar{E},\nabla \cdot {\textbf v}) + 
\langle s\bar{E},{\textbf v}\rangle_{\Gamma_1\cup\Gamma_2} &= 0 \quad 
\mbox{ in } \Omega.
  \end{split}
\end{equation}
Using ${\textbf v} = \bar{E}$ 
(\ref{aux0}), performing the integration over  
$\Omega$ and taking into account the condition that 
$\varepsilon(x)\equiv 0$ in the vicinity of $\Gamma$ hence, 
$(\nabla \cdot( (\varepsilon - 1)\bar{E}),
{\textbf v \cdot \nu})_{\Gamma}  =0$, 
  we get 
\begin{equation}\label{aux1}
\begin{split}
  (s^2  \varepsilon \bar{E} , \bar{E}) 
  + (\nabla  \bar{E},\nabla \bar{E} )  
+ ((\varepsilon  - 1) \nabla \cdot \bar{E},\nabla \cdot \bar{E}) 
+ \langle s\bar{E}, \bar{E}\rangle_{\Gamma_1\cup\Gamma_2} &= 0 \mbox{ in } \Omega,
  \end{split}
\end{equation}
or, equivalently, 
\begin{equation}\label{aux2}
  \parallel  \bar{E} \parallel_{s^2 \varepsilon}^2 + 
  \parallel \nabla  \bar{E} \parallel^2  +
  \parallel \nabla \cdot \bar{E} \parallel_{(\varepsilon - 1)}^2 
	+ \parallel  \bar{E} \parallel_{s,\,\Gamma_1\cup\Gamma_2}^2 = 0.
\end{equation}
From (\ref{aux2})  we see that $\bar{E} \equiv 0$ and hence,  
$\hat{E} = \tilde{E}$, or the solution  $\hat{E}$ 
of the stabilized problem \eqref{eq3bis} is the same as  the solution 
$\tilde{E}$  of the original problem \eqref{model2}.

\subsection{Finite element discretization}

We consider a partition of $\Omega$ into elements $K$
 denoted by ${\mathcal T}_h = \{K\}$, satisfying the standard finite 
element subdivision with the minimal angle condition of elements 
$K\in {\mathcal T}_h$.  
Here, $h=h(x)$ is a mesh function
 defined as $h |_K = h_K$, representing the local diameter of the elements.
We also denote by  $\partial {\mathcal T}_h = \{\partial K\}$ a partition of
 the boundary $\Gamma$ into boundaries  
$\partial K$ of the elements $K$ such that vertices of 
these elements belong to  $\Gamma$. Note that, although $\Gamma$ 
is polygonal (in $2d$) or plans-surface ($3d$) domain, 
it is a priori the case that $\partial K\subset \Gamma$, 
hence to fix this, some 
further subdivision might be necessary. 

To formulate the  finite element method for \eqref{model3} in $\Omega$,
   we introduce the, piecewise linear, finite element space $W_h^E(\Omega)$ 
for every component of the electric  field $E$ defined by
\begin{equation}
W_h^E(\Omega) 
:= \{ w \in H^1(\Omega): w|_{K} \in  P_1(K),  \forall K \in {\mathcal T}_h\}, \nonumber
\end{equation}
where $P_1(K)$ denote the set of piecewise-linear functions on $K$.
 Setting ${\textbf W_h^E(\Omega)} := [W_h^E(\Omega)]^3$ 
we define ${f_0}_h$ to be the usual ${\textbf  W_h^E}$-interpolant of 
 $f_0$ in \eqref{model3}. Then the finite element method for the
 problem \eqref{model3} in $\Omega$ 
  is formulated as: \\
\emph{Find }$\hat E_{h}\in {\textbf  W_h^E(\Omega)}$ 
\emph{ such  that} $\forall {\textbf v} \in {\textbf  W_h^E(\Omega)}$ 
\begin{equation}\label{eq6}
  \begin{split}
 & (s^2  \varepsilon \hat{E}_h, {\textbf v})
    + (\nabla  \hat{E}_h, \nabla {\textbf v})  
    +(\nabla \cdot (\varepsilon  \hat{E}_h), \nabla \cdot {\textbf v}) - (\nabla \cdot \hat{E}_h, \nabla \cdot {\textbf v}) \\ 
  & + \langle s \hat{E}_h, {\textbf v}\rangle_{\Gamma_1\cup\Gamma_2}
  = (s \varepsilon {f_0}_h,  {\textbf v}) + \langle {f_0}_h, {\textbf v}\rangle_{\Gamma_1\cup\Gamma_2}. \;
  \end{split}
\end{equation}
\begin{remark}
Recalling the Assumption 1, the boundary term:
$\langle \nabla\cdot(\varepsilon\hat E)-\nabla\cdot \hat E\rangle_\Gamma$ 
present in  \eqref{eq2} and \eqref{eq3} 
vanishes and hence does not appear in \eqref{eq6} and the subsequent relations.  
\end{remark}
\begin{theo}[well-posedness] \label{Wellposedness1}
The problem \eqref{eq6} has a unique solution 
$\hat{E}_h\in {\textbf  W}_h^E(\Omega)$. 
\end{theo} 

\begin{proof}
We define the bilinear and linear forms, respectively, as 
\begin{equation}\label{bilinearform1}
\begin{split}
a( \hat{E}_h, {\textbf v})= & (s^2  \varepsilon \hat{E}_h, {\textbf v})
    + (\nabla  \hat{E}_h, \nabla {\textbf v})  
    +(\nabla \cdot (\varepsilon  \hat{E}_h), \nabla \cdot {\textbf v}) \\
& - (\nabla \cdot \hat{E}_h, \nabla \cdot {\textbf v}) 
+ \langle s \hat{E}_h, {\textbf v}\rangle_{\Gamma_1\cup\Gamma_2}
\end{split}
\end{equation}
and 
$$
{\mathcal L}( {\textbf v}):= 
(s \varepsilon f_{0,h},  {\textbf v}) 
+ \langle f_{0,h}, {\textbf v}\rangle_{\Gamma_1\cup\Gamma_2},
$$ 
and restate the equation \eqref{eq6} in its compact form as 
\begin{equation}\label{semidiscprob1A}
a( \hat{E}_h, {\textbf v})={\mathcal L}( {\textbf v}).
\end{equation}
Now the well-posedness rely on the Lax-Milgram approach based on 
showing that 
$a( \cdot, \cdot)$ is coercive and both $a( \cdot, \cdot)$ and 
${\mathcal L}( {\cdot})$ are continuous. 
To this end we introduce the triple norm 
$$
\vert\vert\vert \hat{E}_h\vert\vert\vert ^2:= 
\parallel{\hat{E}_h}\parallel_{s^2\varepsilon}^2+
\parallel{\nabla\hat{E}_h}\parallel^2+
\parallel{\nabla\cdot\hat{E}_h}\parallel_{\varepsilon -1}^2
+\parallel{\hat{E}_h}\parallel_{s, \Gamma_1\cup\Gamma_2}^2
$$
and show that there are constants $C_i,\,\, i=1,2,3$ such that 
for all $ \hat{E}_h$ and 
${\textbf v} \in {\textbf  W}_h^E(\Omega)$, 
\begin{align}\label{Coercivity1}
a( \hat{E}_h, \hat{E}_h) &\ge 
C_1 \vert\vert\vert \hat{E}_h\vert\vert\vert ^2,  &&\mbox{(Coercivity)},\\
a( \hat{E}_h, {\textbf v}) &\le C_2 \vert\vert\vert \hat{E}_h\vert\vert\vert
\cdot \vert\vert\vert{\textbf v} \vert\vert\vert, 
&&\mbox{(Continuity of $a$)},\\
\abs{{\mathcal L}( {\textbf v})} &\le  C_3
\vert\vert\vert {\textbf v} \vert\vert\vert, && 
\mbox{(Continuity of ${\mathcal L} $)}
\end{align}
The first relation is straightforward  
and in fact an equality holds with $C_1=1$, viz. 
$$
a( \hat{E}_h, \hat{E}_h)\equiv \vert\vert\vert \hat{E}_h\vert\vert\vert ^2. 
$$
The continuity of $a(\cdot, \cdot)$ is evident from the fact that, 
for every ${\textbf v}\in{\textbf  W_h^E(\Omega)}$, 
 using {\sl Cauchy-Schwarz' inequality}, we have that 
\begin{equation}\label{Coercivity2} 
\begin{split}
a( \hat{E}_h, {\textbf v}) =& 
(s\sqrt\varepsilon  \hat{E}_h, s\sqrt\varepsilon {\textbf v}) +
(\nabla  \hat{E}_h, \nabla {\textbf v}) \\
& +(\nabla \cdot (\sqrt{\varepsilon-1}  \hat{E}_h), 
\nabla \cdot ({\sqrt{\varepsilon-1}\textbf v}))
+\langle \sqrt s \hat{E}_h,  \sqrt s  {\textbf v}
\rangle_{\Gamma_1\cup\Gamma_2}  \\
 \le & 
  \parallel{\hat{E}_h}\parallel_{s^2\varepsilon}
\parallel {\textbf v }\parallel_{s^2\varepsilon}+
    \parallel{\nabla\hat{E}_h}\parallel 
 \parallel\nabla {\textbf v }\parallel 
 +   \parallel{\nabla\cdot\hat{E}_h}\parallel_{\varepsilon -1}
    \parallel{\nabla\cdot {\textbf v} }\parallel_{\varepsilon -1} \\
&+ \parallel{\hat{E}_h}\parallel_{s, \Gamma_1\cup\Gamma_2}
\parallel{\textbf v }\parallel_{s, \Gamma_1\cup\Gamma_2}
\le \vert\vert\vert \hat{E}_h\vert\vert\vert \cdot 
\vert\vert\vert {\textbf v } \vert\vert\vert.
\end{split}
\end{equation}
Likewise, for 
$f_{0,h}\in L_{2,\varepsilon} (\Omega)\cap 
L_{2, 1/{s}}(\Gamma_1\cup\Gamma_2)$ 
we   can easily verify that           
\begin{equation}\label{coercivity3}
\begin{split}
{\mathcal L}( {\textbf v }) &= 
\Big(\sqrt\varepsilon f_{0,h}, s\sqrt\varepsilon {\textbf v }\Big) 
+\langle f_{0,h}, {\textbf v }\rangle_{\Gamma_1\cup\Gamma_2} \\
&\le  
\parallel{f_{0,h}}\parallel_{\varepsilon}
\parallel{\textbf v}\parallel_{s^2\varepsilon}+
\parallel{f_{0,h}}\parallel_{ 1/{s} ,\Gamma_1\cup\Gamma_2}
\parallel{\textbf v}\parallel_{s, \Gamma_1\cup\Gamma_2}\\
&\le \Big( \parallel{f_{0,h}}\parallel_{\varepsilon}+
\parallel{f_{0,h}}\parallel_{ 1/{s} ,\Gamma_1\cup\Gamma_2}\Big)\,
\vert\vert\vert {\textbf v } \vert\vert\vert,
\end{split}
\end{equation}
and hence ${\mathcal L}$ is continuous as well.
Summing up, by the coercivity 
\eqref{Coercivity1} and the continuities \eqref{Coercivity2} and 
\eqref{coercivity3} we may use  
the Lax-Milgram theorem which guarantee's the existence of a unique 
solution for the discrete problem \eqref{eq6}. The continuities would 
yield even stability and hence justifies the well-posedness and completes 
the proof. 

\end{proof} 

\section{Error analysis}
In this section first we give a swift a priori error bound and then 
continue with a posteriori error estimates and derive  a corresponding 
adaptive algorithm for the above, piecewise 
linear approximation of the time harmonic Maxwell's equations formulated in 
\eqref{eq6}. 

\subsection{A priori error estimates}
To derive a priori error estimates we need to define a continuous 
version of the linear and bilinear forms introduced in well-posedness 
theorem. To this end, we rewrite \eqref{bilinearform1} replacing  
$\hat{E}_h$ by $\hat{E}$ and 
with a new continuous linear form 
defined with same expression as ${\mathcal L}({\textbf v})$, but for 
${\textbf v}\in H^1(\Omega)$ rather than in the finite element space 
${\textbf W}^E_h(\Omega)$: 
\begin{equation}\label{contbilinear1}
\begin{split}
a(\hat{E}, {\textbf v})= & (s^2  \varepsilon \hat{E}, {\textbf v})
    + (\nabla  \hat{E}, \nabla {\textbf v})  
    +(\nabla \cdot (\varepsilon  \hat{E}), \nabla \cdot {\textbf v}) \\
& - (\nabla \cdot \hat{E}, \nabla \cdot {\textbf v}) 
+ \langle s \hat{E}, {\textbf v}\rangle_{\Gamma_1\cup\Gamma_2},\qquad \forall {\textbf v}\in H^1(\Omega)
\end{split}
\end{equation}
and 
\begin{equation}\label{contbilinear2}
{\mathcal L}^c({\textbf v}):=(s\varepsilon f_0,  {\textbf v})+ 
\langle f_0, {\textbf v}\rangle_{\Gamma_1\cup\Gamma_2}, \qquad \forall {\textbf v}\in H^1(\Omega).
\end{equation}
Hence we have the concise form of the variational formulation 
\begin{equation}\label{contprob1A}
a(\hat{E}, {\textbf v})={\mathcal L}^c({\textbf v}), 
\qquad \forall {\textbf v}\in H^1(\Omega).
\end{equation} 
To proceed we derive {\sl Galerkin orthogonality} relation by letting, in 
\eqref{contbilinear1} and \eqref{contbilinear2}, 
${\textbf v}\in {\textbf W}^E_h(\Omega)$, as well as replacing 
$f_0$ by $f_{0,h}$ in  \eqref{contbilinear2}. 
Subtracting \eqref{semidiscprob1A} from the, such obtained, 
continuous problem 
\eqref{contprob1A} and letting 
$e(x, s):= \hat{E}(x,s)-\hat{E}_h (x,s)$ be the pointwise spatial 
error of the finite element approximation \eqref{eq6}, we get 

\begin{equation}\label{GOrtho} 
a(\hat{E}-\hat{E}_h, {\textbf v})=0,\qquad \forall 
{\textbf v}\in {\textbf W}_h^E(\Omega), 
\qquad (\mbox{Galerkin orthogonality}). 
\end{equation}
Note that restricting \eqref{contprob1A} to ${\textbf v}\in {\textbf W}_h^E(\Omega)$, 
\eqref{semidiscprob1A} and  \eqref{contprob1A} get the same right hand sides. 

Now we are ready to derive  the 
following theoretical error bound

\begin{theo}\label{Apriori} 
 Let $\hat{E}$ and $\hat{E}_h$ be the solutions for the continuous 
problem \eqref{model3} 
(in the variational form \eqref{eq2})  and its finite element approximation, 
\eqref{eq6}, respectively. 
Then, there is a constant $C$, independent of $\hat E$ and $h$, such that 
$$
\vert\vert\vert e\vert\vert\vert \le C 
\parallel{h \hat{E}}\parallel_{H^2_w(\Omega)}.  
$$
where $w=w(\varepsilon(x), s)$ is the weight function which 
depends on the dielectric permittivity 
function $\varepsilon(x)$ and the pseudo-frequency variable $s$. 
\end{theo}

\begin{proof}
Using the definition of the triple norm and Galerkin orthogonality 
\eqref{GOrtho} we have 
\begin{equation}\label{AprioriEE1}
\begin{split}
\vert\vert\vert e\vert\vert\vert^2 = a(e,e)
=a(e, \hat{E}-\hat{E}_h)=a(e, \hat{E})=a(e, \hat{E}-{\textbf v}),
\qquad\forall {\textbf v}\in {\textbf W}_h^E(\Omega).
\end{split} 
\end{equation}
Now we set ${\textbf v}:=\pi_h\hat{E}$, the interpolant of $\hat{E}$. 
Then, using the Cauchy-Schwarz' inequality and the interpolation 
error estimates we can estimate the right 
hand side of \eqref{AprioriEE1} as follows: 
\begin{equation}\label{AprioriEE2}
\begin{split}
\vert\vert\vert e\vert\vert\vert^2 =&
(e, \hat{E}-\pi_h\hat{E})_{s^2\varepsilon}+
(\nabla e, \nabla(\hat{E}-\pi_h\hat{E}))\\
&+
(\nabla \cdot e, \nabla\cdot(\hat{E}-\pi_h\hat{E}))_{\varepsilon -1}+
\langle e, \hat{E}-\pi_h\hat{E}\rangle_{s,\Gamma_1\cup\Gamma_2}\\
\le &\parallel{e}\parallel_{s^2\varepsilon}
\parallel{\hat{E}-\pi_h\hat{E}}\parallel_{s^2\varepsilon}+
\parallel{\nabla e}\parallel
\parallel{\nabla(\hat{E}-\pi_h\hat{E})}\parallel \\
&+
\parallel{\nabla\cdot e}\parallel_{\varepsilon-1}
\parallel{\nabla\cdot(\hat{E}-\pi_h\hat{E})}\parallel_{\varepsilon-1}+
\parallel{e}\parallel_{s,\Gamma_1\cup\Gamma_2}
\parallel{\hat{E}-\pi_h\hat{E}}\parallel_{s,\Gamma_1\cup\Gamma_2}\\
\le &C_1\parallel{e}\parallel_{s^2\varepsilon}
\parallel{h^2 D^2\hat{E}}\parallel_{s^2\varepsilon}+
C_2\parallel{\nabla e}\parallel
\parallel{h D^2\hat{E}}\parallel +\\
& +C_2\parallel{\nabla\cdot e}\parallel_{\varepsilon-1}
\parallel{h D^2\hat{E}}\parallel_{\varepsilon-1}+\\
&+ C_3\parallel{e}\parallel_{s,\Gamma_1\cup\Gamma_2}
\Big(\parallel{h^2 D^2\hat{E}}\parallel_{s, L_2(\Omega)}^{1/2}\cdot
\parallel{h D^2\hat{E}}\parallel_{s,L_2(\Omega)}^{1/2}\Big) \\
\le &C \vert\vert\vert e\vert\vert\vert\cdot 
 \Big(\parallel{h^2 D^2\hat{E}}\parallel_{s^2\varepsilon}+
\parallel{h D^2\hat{E}}\parallel
+\parallel{h D^2\hat{E}}\parallel_{\varepsilon-1}+\\
&
\qquad\qquad \qquad+{\parallel{h^{3/2} D^2\hat{E}}\parallel}_s\Big)
\le C \vert\vert\vert e\vert\vert\vert\cdot 
\parallel{h D^2\hat{E}}\parallel_w,
\end{split} 
\end{equation}
where $w=\max(h s^2\varepsilon, h^{1/2}(\varepsilon-1), h^{1/2}s)$, and 
to estimate the interpolation error at the boundary 
we apply the Poincare inequality (since  $\hat{E}-\pi_h\hat{E}=0$ 
on the whole or part of $\Gamma$ with positive Lebesgue measure) 
to the  the trace estimate (see Brenner-Scott \cite{Brenner}; 
Theorem 1.6.6): 
\begin{equation}\label{Tracs1}
\begin{split}
\parallel{\hat{E}-\pi_h\hat{E}}\parallel_{s,\Gamma} &\le \tilde C 
\parallel{\hat{E}-\pi_h\hat{E}}\parallel_{s, L_2(\Omega)}^{1/2}
\parallel{\hat{E}-\pi_h\hat{E}}\parallel_{s, H^1(\Omega)}^{1/2}\\
 &\le \tilde C 
\parallel{\hat{E}-\pi_h\hat{E}}\parallel_{s, L_2(\Omega)}^{1/2} 
\parallel{\nabla(\hat{E}-\pi_h\hat{E})}\parallel_{s, L_2(\Omega)}^{1/2}\\
 &\le \tilde C \parallel{h^2 D^2 \hat{E}}\parallel_{s, L_2(\Omega)}^{1/2}\cdot 
\parallel{h D^2 \hat{E}}\parallel_{s, L_2(\Omega)}^{1/2}
\end{split}
\end{equation}
where $D^2$ stands for the differential operator which, e.g., for the 
 2-dimensional domain $\Omega$, i.e., in 
$xy$-geometry and for $u\in{\mathcal C}^{(2)}(\Omega)$   
is given by 
$$
D^2 u:=(u_{xx}^2+2u_{xy}^2+u_{yy}^2)^{1/2}. 
$$
Finally $C=\max C_i, \, i=1,2$ are the interpolation constants 
and $C_3=\tilde C(C_1+C_2)$ and the proof is complete. 
\end{proof}
\begin{remark}
We note that the obtained 
error is of order ${\mathcal O}(h)$ in $H^2_w(\Omega) $ 
which is optimal due to 
the gradient term in the triple norm. 
This may be further improved, to achieve 
superconvergence and estimations in 
the negative norm which in scalar form is defined as 
$$
\norm{w}_{H^{-s}(\Omega)}=\sup_{\varphi\in H^{s}(\Omega)\cap H^1_0(\Omega)}
\frac{(w,\varphi)}{\norm{\varphi}_{H^s(\Omega)}}. 
$$
This, however, extended to vector form, 
yields extra computational complexity. 
Otherwise, {\sl with no gradient term}, 
using the Laplace transform, the equation is rendered from being 
hyperbolic to an elliptic one and the trace theorem estimating the boundary 
integrals is the term setting an upper limit of ${\mathcal O}(h^{3/2})$. 
This however, would not survive due to the fact that the triple norm 
involves gradient norm and the optimal order is indeed ${\mathcal O}(h)$,
(rather than ${\mathcal O}(h^2)$ of the purely $L_2$-norm). 
\end{remark}

\subsection{A posteriori error estimates}
For the approximate solution $ \hat{E}_h= \hat{E}_h(x,s)$ 
of the problem \eqref{model3}, where 
$ x \in \Omega\subset \mathbb{R}^{d},\,\, d=2,3$, 
we define the residual errors, viz. 
\begin{equation}\label{Resid1}
\begin{split}
-{\mathcal R}( \hat{E}_h):=&
 s^2  \varepsilon(x) \hat{E}_h   - \triangle \hat{E}_h
  -  \nabla ( \nabla \cdot (( \varepsilon(x)  -1) \hat{E}_h) -
  s \varepsilon(x)  f_{0,h}(x),\quad\mbox{and}\\
-{\mathcal R}_\Gamma( \hat{E}_h):=&
h^{-\alpha}\Big(\partial _{\nu } \hat{E}_h +s  \hat{E}_h- f_{0,h}(x)\Big), 
\quad\mbox{for} \quad x \in \Gamma_1 \cup \Gamma_2, \quad 0<\alpha\le 1. 
\end{split}
\end{equation} 
By the Galerkin orthogonality we have that 
${\mathcal R}( \hat{E}_h)\perp{\mathbf W}_h^E(\Omega)$. 
Now the objective is to bound the triple norm of the 
error $e(x,s):=\hat{E}(x,s)-  \hat{E}_h(x,s)$ 
by some adequate norms of ${\mathcal R}( \hat{E}_h) $ 
and ${\mathcal R}_\Gamma( \hat{E}_h)$ with a relevant, fast, decay. 
This may be done in a few, relatively similar, ways, e.g., one 
can use the variational formulation and interpolation in the error 
combined with Galerkin orthogonality.  
Or one may use a dual problem approach setting the source term 
(or initial data) 
as the error. We check 
both techniques to the time harmonic Maxwell's equations. 
The former is given in the Theorem \ref{TheoApost1} below 
and the latter is the subject of Proposition \ref{Prop1Apost}. 
The proof of the 
main theorem will rely on assuming a first order approximation for the 
initial value of the original field $f_0(x):=E(x,t)\vert_{t=0_-}$, 
for $\beta\approx 1$, and according to 
\begin{equation}\label{Error1BB}
\parallel f_0-f_{0,h}\parallel_\varepsilon\approx 
\parallel f_0-f_{0,h}\parallel_{1/s, \Gamma}\approx 
\parallel f_0-f_{0,h}\parallel_{(\varepsilon -1)^2/s, \Gamma}
={\mathcal O}(h^\beta).
\end{equation}
The justification for \eqref{Error1BB} is that $f_0$, as the time independent 
field $E(x,0)$  satisfies a Poisson type equation and, tackling the 
contributions from the weights, a finite element
bound of the form  \eqref{Error1BB} is evident. 

\begin{theo}\label{TheoApost1} 
Let $\hat{E}$ and $\hat{E}_h$ be the solutions for the 
continuous problem  \eqref{model3} 
(in the variational form \eqref{eq2}) and its finite element approximation 
\eqref{eq6}, respectively. Further we assume that we have the error bound 
\eqref{Error1BB} for the initial field  $f_0(x):=E(x,t)\vert_{t=0_-}$. 
Then, there exist interpolation constants $C_1$ and $C_2$, 
independent of $h,\, $ and $\hat{E}$,  but may depend on 
$\varepsilon$ and $s$
such that the following a posteriori 
error estimate holds true
\begin{equation}\label{Resid2}
\vert\vert\vert e\vert\vert\vert\le 
C_1\, h\parallel {\mathcal R}\parallel+ 
C_2\, h^{\alpha}
\parallel {\mathcal R}_\Gamma\parallel_{1/s,\,\Gamma_1\cup\Gamma_2}+\,
{\mathcal O}(h^\beta), 
\end{equation} 
where $\alpha\approx\beta\approx 1$. 
\end{theo}

\begin{proof}

We shall use the weak form emerging from the triple norm of the error 
where the contribution from the $\partial_\nu$ 
 will be canceled using Green's formula to the 
Laplacian 
\begin{equation}\label{Error1AA}
\begin{split}
\vert\vert\vert e \vert\vert\vert^2
= &
 (s^2\varepsilon e,\, e)+(\nabla e,\, \nabla e)
 +(\nabla\cdot ((\varepsilon-1)e), 
\nabla\cdot e)+\langle se,\, e\rangle_{\Gamma_1\cup\Gamma_2} \\
&-\langle  \nabla\cdot(\varepsilon e), e\rangle_\Gamma 
+\langle  \nabla\cdot(\varepsilon e), e\rangle_\Gamma 
:=\sum_{j=1}^4 I_j. 
\end{split} 
\end{equation}
Now by the Green's formula and the fact that 
$\partial_\nu e\equiv 0$ on $\Gamma_3$, we have 
\begin{equation}\label{Error1AB}
\begin{split}
I_2+I_3 =& 
 -(\Delta e, e) 
-\Big(\nabla(\nabla\cdot((\varepsilon -1)e)),e\Big)\\
&+\langle \partial_\nu e,e\rangle_{\Gamma_1\cup\Gamma_2}
+\langle \partial_\nu((\varepsilon -1)e),e\rangle_{\Gamma_1\cup\Gamma_2}\\
=& -(\Delta e, e) 
-\nabla(\nabla\cdot((\varepsilon -1)e),e) 
+\langle \partial_\nu(\varepsilon e),\, e\rangle_{\Gamma_1\cup\Gamma_2}.
\end{split} 
\end{equation}
Hence 
\begin{equation}\label{Error1AC}
\begin{split}
\vert\vert\vert e \vert\vert\vert^2= &
 (s^2\varepsilon e,\, e) -(\Delta e, e) 
-\nabla(\nabla\cdot((\varepsilon -1)e),e) \\
&+\langle\partial_\nu (\varepsilon e),\, e\rangle_{\Gamma_1\cup\Gamma_2}
+\langle s e, e\rangle_{\Gamma_1\cup\Gamma_2}
\end{split} 
\end{equation}
Thus, recalling the definition of the residuals 
${\mathcal R}$ and ${\mathcal R_\Gamma}$, we can write 
\begin{equation}\label{Error1A}
\begin{split}
\vert\vert\vert e \vert\vert\vert^2=&
 (s^2\varepsilon \hat{E}-\Delta \hat{E}- 
\nabla(\nabla\cdot((\varepsilon-1)\hat{E}),\, e)+
\langle \partial_\nu(\varepsilon \hat E), e\rangle_{\Gamma_1\cup\Gamma_2}
+\langle s\hat{E},\, e\rangle_{\Gamma_1\cup\Gamma_2}\\
&-( s^2\varepsilon \hat{E}_h-\Delta_h \hat{E}- 
\nabla_h(\nabla\cdot((\varepsilon-1)\hat{E}),\, e)
-\langle \partial_\nu(\varepsilon \hat {E}_h), e\rangle_{\Gamma_1\cup\Gamma_2}
-\langle s\hat{E}_h,\, e\rangle_{\Gamma_1\cup\Gamma_2}\\
=&(s\varepsilon f_0, \, e)+({\mathcal R}(\hat{E}_h),\, e)- 
(s\varepsilon f_{0,h}, \, e)
+\langle\partial_\nu((\varepsilon-1)\hat{E}), e\rangle_{\Gamma_1\cup\Gamma_2} 
+\langle\partial_\nu\hat{E}, e\rangle_{\Gamma_1\cup\Gamma_2}\\ 
&+\langle s\hat{E},\, e\rangle_{\Gamma_1\cup\Gamma_2} 
-\langle\partial_\nu((\varepsilon-1)\hat{E}_h), e\rangle _{\Gamma_1\cup\Gamma_2} 
-\langle\partial_\nu\hat{E}_h, e\rangle_{\Gamma_1\cup\Gamma_2}- \langle s\hat{E}_h,\, e\rangle_{\Gamma_1\cup\Gamma_2}\\
=& (s\varepsilon (f_0-f_{0,h}), \, e)+({\mathcal R}(\hat{E}_h),\, e)
+\langle\partial_\nu((\varepsilon-1)(\hat{E}-\hat{E_h})), e\rangle _{\Gamma_1\cup\Gamma_2}\\
&+\langle{h^{\alpha}\mathcal R}_{\Gamma}(\hat{E}_h), \, e\rangle_{\Gamma_1\cup\Gamma_2}
+\langle f_0-f_{0,h}, e\rangle_{\Gamma_1\cup\Gamma_2}, 
\end{split}
\end{equation}
where the contributions from the $\Gamma_3$ to the boundary terms are all 
identically zero. 
Further, we use the orthogonality relation 
${\mathcal R}(\hat{E}_h)\perp {\mathbf W}_h^E(\Omega)$, so that 
\begin{equation}\label{Error1AA}
\begin{split}
\vert\vert\vert e \vert\vert\vert^2= &
 (\sqrt \varepsilon( f_0-f_{0,h}), \,s\sqrt\varepsilon e)+
({\mathcal R}(\hat{E}_h),\, e-\pi_h e)\\
&+\langle 
\partial_\nu((\varepsilon-1)(\hat{E}-\hat{E_h})),  e\rangle _{\Gamma_1\cup\Gamma_2}
+\langle h^{\alpha}{\mathcal R}_{\Gamma}(\hat{E}_h), \, 
e\rangle_{\Gamma_1\cup\Gamma_2}
\\
&+
\langle \frac 1{\sqrt s}( f_0-f_{0,h}),\sqrt s e\rangle_{\Gamma_1\cup\Gamma_2}:=\sum_{j=1}^5 J_j
\end{split}
\end{equation}
Now we estimate each $J_j$-term separately 
$$
J_1:=  (\sqrt \varepsilon( f_0-f_{0,h}), \,s\sqrt\varepsilon e) 
\le \parallel  f_0-f_{0,h}\parallel_\varepsilon \cdot 
\parallel e\parallel_{\varepsilon s^2}
$$
Using the interpolation estimates 
\begin{equation}\label{J2J3}
\begin{split}
J_2+J_4:= &({\mathcal R}(\hat{E}_h),\, e-\pi_h e)+ 
\langle h^{\alpha}{\mathcal R}_{\Gamma}(\hat{E}_h), \, 
e\rangle_{\Gamma_1\cup\Gamma_2}\\ 
\le & 
\Big(\parallel{h \mathcal R}(\hat{E}_h)\parallel \parallel \nabla e \parallel + 
\parallel{h^\alpha \mathcal R}_{\Gamma}(\hat{E}_h)\parallel_{1/s,\Gamma_1\cup\Gamma_2} 
 \parallel e \parallel_{s,\Gamma_1\cup\Gamma_2} \Big). 
\end{split}
\end{equation}
Finally 
\begin{equation}\label{J5}
J_5:= \langle \frac 1{\sqrt s}( f_0-f_{0,h})+
\sqrt s e\rangle _\Gamma 
\le 
\parallel f_0-f_{0,h} \parallel_{1/s. \Gamma_1\cup\Gamma_2}
\parallel e \parallel_{s,\Gamma_1\cup\Gamma_2}.
\end{equation}
It remains to bound the term $J_3$, to this end we use the fact that 
$\partial_\nu\varepsilon=0$ and hence 
\begin{equation*}\label{J3}
\begin{split}
J_3 &=\langle 
\partial_\nu((\varepsilon-1)(\hat{E}-\hat{E_h})),  e\rangle _\Gamma 
=\langle (\partial_\nu \varepsilon)(\hat{E}-\hat{E_h}),  e\rangle _\Gamma +
\langle (\varepsilon-1)\partial_\nu(\hat{E}-\hat{E_h}), e\rangle _\Gamma\\
&=\langle (\varepsilon-1)\partial_\nu(\hat{E}-\hat{E_h}), e\rangle _\Gamma.
\end{split}
\end{equation*}
Using both continuous and discrete boundary conditions we have 
$$
\partial_\nu(\hat{E}-\hat{E_h})=(f_0-f_{0,h})-s(\hat{E}-\hat{E_h})\qquad 
\mbox{on}\quad \Gamma_1\cup\Gamma_2, 
$$
and 
$$
\partial_\nu(\hat{E}-\hat{E_h})=0\qquad 
\mbox{on}\quad \Gamma_3, 
$$
we can estimate $J_{3}$ as follows      
\begin{equation}\label{J32}
\begin{split}
J_{3}&= 
\langle (\varepsilon-1)\partial_\nu(\hat{E}-\hat{E_h}), e\rangle_{\Gamma_1\cup\Gamma_2}\\
&= \langle (\varepsilon-1)( f_0-f_{0,h})\frac 1{\sqrt s}, 
{\sqrt s} e\rangle_{\Gamma_1\cup\Gamma_2}-  
\langle (\varepsilon-1)s e, e\rangle_{\Gamma_1\cup\Gamma_2} \\
&\le \parallel f_0-f_{0,h}\parallel_{(\varepsilon -1)^2/s, \Gamma_1\cup\Gamma_2}
\parallel e \parallel_{s,\Gamma_1\cup\Gamma_2} - 
\min \abs{\varepsilon-1}\parallel e \parallel_{s,\Gamma_1\cup\Gamma_2}. 
\end{split}
\end{equation}
Summing up we have 
\begin{equation}\label{Summing1}
 \begin{split}
 \vert\vert\vert e \vert\vert\vert^2 \le &
 \parallel h {\mathcal R}(\hat{E}_h)\parallel \parallel \nabla e \parallel
+\parallel h^\alpha {\mathcal R}_\Gamma(\hat{E}_h)
\parallel_{1/s, \Gamma_1\cup\Gamma_2} 
\parallel e \parallel_{s,\Gamma_1\cup\Gamma_2} \\
&+ \parallel f_0-f_{0,h} \parallel_\varepsilon 
\parallel e\parallel_{\varepsilon s^2}+
\parallel f_0-f_{0,h} \parallel_{1/s, \Gamma_1\cup\Gamma_2} 
\parallel e\parallel_{s,\Gamma_1\cup\Gamma_2}\\
&+  \parallel f_0-f_{0,h}\parallel_{(\varepsilon -1)^2/s, \Gamma_1\cup\Gamma_2}
\parallel e \parallel_{s,\Gamma_1\cup\Gamma_2} - 
\min \abs{\varepsilon-1}\parallel e \parallel_{s,\Gamma_1\cup\Gamma_2}.
\end{split}
\end{equation}
Now recalling  \eqref{Error1BB}, 
by a ``kick-back'' of the negative term 
$-\min \abs{\varepsilon-1}\parallel e \parallel_{s,\Gamma}$ or ignoring it 
we end up with 
\begin{equation}\label{Error1AE} 
\vert\vert\vert e \vert\vert\vert^2\le  
C\Big(\parallel{ h\mathcal R}(\hat{E}_h)\parallel+
\parallel h^{\alpha}{\mathcal R}_\Gamma(\hat{E}_h)
\parallel_{1/s, \Gamma_1\cup\Gamma_2} +{\mathcal O}(h^\beta)\Big)
\vert\vert\vert e \vert\vert\vert.
\end{equation}
This yields the desired result 
\begin{equation}\label{Error1AF}
\vert\vert\vert e \vert\vert\vert\le 
C\Big(\parallel{ h\mathcal R}(\hat{E}_h)\parallel+
\parallel h^{\alpha}{\mathcal R}_\Gamma(\hat{E}_h)
\parallel_{1/s,\Gamma_1\cup\Gamma_2} +{\mathcal O}(h^\beta)\Big), 
\end{equation} 
where assuming $\alpha\approx\beta\approx 1$, \eqref{Error1AF} 
is optimal. 
\end{proof}
\subsection{A duality approach} 
In Theorem \ref{Apriori} 
 we made a direct a priori error estimate. 
The a posteriori error estimates in Theorem \ref{TheoApost1} 
rely on bounds depending on the residual of the 
approximate solution,  
with no reference to the dual problem. There is yet another 
approach based on an adequate 
dual form of the original problem. This can be studied both in 
a priori as well as a posteriori regi. 
Below we give a version of the a priori estimate: Theorem \ref{Apriori}  
based on a dual problem formulation. This approach is quite similar to 
that of the proof of Theorem \ref{Apriori} and is included in here 
just in order to show 
the dual technique, as an alternative approach for the cases when 
establishing a reasonable convergence via Theorem \ref{Apriori} becomes too expensive. 
Here we assume that the dual problem has a sufficiently smooth solution 
in the weighted Sobolev space with a norm that naturally arises along the
estimates. Now, we define the differential operator ${\mathcal A}$ 
by
$$
({\mathcal A} \varphi, {\textbf v})=a( \varphi,  {\textbf v}), \qquad \forall 
\varphi,\,\, {\textbf v}\in W_h^E(\Omega), 
$$
and let ${\mathcal A}:={\mathcal A}_\Omega +{\mathcal A}_\Gamma$ where 
$$
{\mathcal A}_\Omega \varphi:= 
s^2\varepsilon (x)\varphi(x,s)-\Delta \varphi(x,s)
-\nabla(\nabla\cdot(\varepsilon(x)-1)\varphi(x,s)),\quad x\in \Omega
$$
is a self adjoint operator, whereas 
$$
{\mathcal A}_\Gamma\varphi(x,s):=\partial_\nu (\varepsilon (x)\varphi(x,s))+s \varphi(x,s),\qquad 
x\in \Gamma_1\cup\Gamma_2,
$$
is non-self adjoint with the adjoint ${\mathcal A}_\Gamma^\prime$ given by 
$$
{\mathcal A}_\Gamma^\prime\varphi(x,s):=
-\partial_\nu (\varepsilon (x)\varphi(x,s))+s \varphi(x,s),\qquad 
x\in \Gamma_1\cup\Gamma_2.
$$
To proceed 
we consider the dual problem 
\begin{equation}\label{Dual1}
\left\{
\begin{array}{ll}
\displaystyle {\mathcal A}_\Omega \varphi (x,s)=e(x,s),\qquad & x\in \Omega\\\\
\displaystyle {\mathcal A}_\Gamma^\prime\varphi(x,s)=
-\partial_\nu (\varepsilon (x)e(x,s))+ se(x,s), & x\in  \Gamma_1\cup\Gamma_2\\\\
\displaystyle \partial_\nu\varphi(x,s)=0,\qquad & x\in \Gamma_3. 
\end{array}
\right.
\end{equation}
\begin{proposition}[Error estimate using dual problem] \label{Prop1Apost}    
Let $\varphi\in{\mathbf W}_E^2:=H^2_w (\Omega)\cap H^2_{1/s}(\Gamma)$ in the 
above duality formulation. Then, we have the following error estimate: 
$$
 \vert\vert\vert e \vert\vert\vert\le 
h^2\parallel\varphi\parallel_{H^2_{w(\Omega)}}+h^{3/2} 
\parallel\varphi\parallel_{H^2_{1/s}(\Gamma)},
$$              
\end{proposition}
\begin{proof}
For
$\varphi\in{\mathbf W}_E^2$  and  
${\textbf v}$ the nodal interpolant of $\varphi$ in a finite element 
partition ${\mathcal T}_h$ of the computational domain $\Omega$, we have by a multiple 
use of Green's formula and 
a Galerkin orthogonality relation that 
\begin{equation}\label{Error1D}
\begin{split}
\vert\vert\vert e \vert\vert\vert^2= & 
\parallel{e}\parallel_{s^2\varepsilon}^2+
\parallel{\nabla e}\parallel^2+
\parallel{\nabla\cdot e}\parallel_{\varepsilon -1}^2
+\parallel{e}\parallel_{s, \Gamma_1\cup\Gamma_2}^2\\ 
=& ({\mathcal A}_\Omega \varphi, e)+ ( {\mathcal A}_\Gamma\varphi, e)\\
=& (\varphi, e)_{s^2\varepsilon}+
 (\nabla \varphi, \nabla e)+
(\nabla \cdot( (\varepsilon -1)) \varphi, \nabla\cdot e)
+\langle\varphi, e\rangle_{s,\Gamma_1\cup\Gamma_2}\\
=& (\varphi-{\textbf v} , e)_{s^2\varepsilon}+ 
 (\nabla (\varphi -{\textbf v}), \nabla e)\\
&+
(\nabla \cdot ((\varepsilon -1)  \varphi-{\textbf v}), \nabla\cdot e)
+\langle \varphi-{\textbf v} , e\rangle_{s,\Gamma_1\cup\Gamma_2}\\
\le & \parallel h^2 D^2\varphi\parallel_{s^2\varepsilon}
\parallel{e}\parallel_{s^2\varepsilon}+
\parallel{h D^2\varphi}\parallel \parallel{\nabla e}\parallel \\
&+
\parallel{h D^2\varphi}\parallel_{\varepsilon -1}  
\parallel{\nabla\cdot e}\parallel_{\varepsilon -1} 
 +\parallel{h^{3/2} D^2\varphi}\parallel_{s, \Gamma_1\cup\Gamma_2}
\parallel{e}\parallel_{s, \Gamma_1\cup\Gamma_2}\\
\le & 
\Big(\parallel h^2 D^2\varphi\parallel_{s^2\varepsilon}
+\parallel{h D^2\varphi}\parallel 
+\parallel{h D^2\varphi}\parallel_{\varepsilon -1}  
+\parallel{h^{3/2} D^2\varphi}\parallel_{s, \Gamma_1\cup\Gamma_2}\Big)
\vert\vert\vert{e}\vert\vert\vert.
\end{split}
\end{equation}
Thus 
$$
\vert\vert\vert e \vert\vert\vert\le 
\parallel h \varphi\parallel_{H^2_{w(\Omega)}}+
\parallel h^{3/2} \varphi\parallel_{H^2_{s}(\Gamma)},
$$
which, is exactly the same estimate as 
in \eqref{AprioriEE2} and  \eqref{Error1AF}. Hence, the improved estimate 
(by ${\mathcal O}(h^{1/2})$)
 of order ${\mathcal O}(h^{3/2})$, due to the trace theorem is not helpful, 
due to the fact that the triple norm involves $L_2$-norm of the gradient 
of the error and the final, optimal, estimate is of ${\mathcal O}(h)$. 
\end{proof}
\subsection{Adaptivity algorithm}  
In this part we outline the adaptivity algorithm 
for the a posteriori error estimate given by
\eqref{Error1AF}. To this end we consider an {\sl error tolerance} TOL 
and seek to construct a discrete scheme that would guarantee the error bound 
\begin{equation}\label{Adaptiv1}
\vert\vert\vert e \vert\vert\vert\le \mbox{TOL}.
\end{equation}
A concise adaptivity procedure is introduced as in the following steps
\begin{itemize}

\item[I] Given the dielectric permittivity function $\varepsilon (x)$,  
the initial data $f_0$ and the corresponding boundary conditions as in 
\eqref{E_gauge} and \eqref{2.3}. 

\item[II]
Consider a mesh size $h$, choose 
$\alpha\approx\beta\approx 1$ and 
compute the 
approximate initial data $f_{0,h}$ and the corresponding approximate solution 
$\hat{E}_h$, using the finite element scheme \eqref{eq6}.

\item[III] Use the definition of the residuals, viz. \eqref{Resid1} and compute 
${\mathcal R}(\hat{E}_h)$ and ${\mathcal R}_\Gamma(\hat{E}_h)$

\item[IV]
If 
\begin{equation}\label{Error1Adapt}
\vert\vert\vert e \vert\vert\vert\le 
\Big(\parallel{ h\mathcal R}(\hat{E}_h)\parallel+
\parallel h^{\alpha}{\mathcal R}_\Gamma(\hat{E}_h)
\parallel_{1/s,\Gamma_1\cup\Gamma_2} +{\mathcal O}(h^\beta)\Big)\le \mbox{TOL}, 
\end{equation} 
(observe that we hide the constant $C$ in TOL)
then stop and accept the approximate solution $\hat{E}_h$. 
Otherwise, refine the 
mesh in the parts of the domain $\Omega$ and the boundary $\Gamma$ where 
either or both of ${\mathcal R}(\hat{E}_h)$ and 
${\mathcal R}_\Gamma(\hat{E}_h)$
are large and go to step II with this new mesh parameter.   
\end{itemize}

\section{Numerical examples}

We perform numerical tests  in the computational domain $\Omega =
[0,1] \times [0,1]$, with 
$\Omega_1:=[0.25,\,0.75] \times [0.25,\, 0.75]$. 
 The source data (the right hand side) in the model problem (2.4) is chosen such that the function
\begin{equation}\label{eq1}
  \begin{split}
    E_1 &=    \frac{1}{\varepsilon} 2 \pi \sin^2 \pi x  \cos \pi y  \sin \pi y ~ \frac{t^2}{2}, \\
   E_2 &=  - \frac{1}{\varepsilon}   2 \pi   \sin^2 \pi y \cos \pi x   \sin \pi x   ~\frac{t^2}{2}
  \end{split}
\end{equation}
is the exact solution of this problem. After application of the 
Laplace transform (2.2) to  \eqref{eq1}
the exact solution of the transformed model problem (2.8) will be
\begin{equation}\label{modp2_1}
  \begin{split}
    E_1 &=    \frac{2}{ s^3 \varepsilon} \pi \sin^2 \pi x  \cos \pi y  \sin \pi y , \\
   E_2 &=  - \frac{2}{ s^3 \varepsilon}    \pi   \sin^2 \pi y \cos \pi x   \sin \pi x. 
  \end{split}
\end{equation}

In \eqref{eq1} the function $\varepsilon$  is defined for an integer $m > 1$ as
\begin{equation}\label{eps}
  \varepsilon(x,y)= \left \{
  \begin{array}{ll}
    1 + \sin^m \pi (2x-0.5) \cdot \sin^m \pi (2y-0.5) & 
\textrm{in $\Omega_1$}, \\
    1 &  \textrm{in $\Omega_2:=\Omega\setminus\Omega_1$}.
  \end{array}
  \right.
\end{equation}
The solution defined in \eqref{eq1}  satisfies the homogeneous initial conditions. We chose homogeneous boundary conditions at the boundary $\Gamma=\partial\Omega$
 of the computational domain $\Omega$.
 Figure \ref{fig:F1}   shows the function $\varepsilon$ for different  values of $m$.

 We discretize the computational domain $\Omega\times (0,T)$ denoting
 by ${{\mathcal T}_h}_l = \{K\}$ a partition of the domain $\Omega$ into
 triangles $K$ of sizes $h_l= 2^{-l}, l=1,...,6$.  Numerical tests are
 performed for different $m=2,..., 9$ in \eqref{eps} and the relative
 errors are then measured in $L_2$-norm and the
 $H^1$-norms, respectively, which are computed as
\begin{align}
e_l^1 &=  \frac{\|E-E_h\|_{L_2}}{\| E\|_{L_2}},\\
e_l^2 &=  \frac{\|\nabla (E-E_h) \|_{L_2}}{\| \nabla E\|_{L_2}},  
\end{align}
where, as well as in the sequel 
\begin{equation}\label{AbsE}
\abs{E}:=\sqrt{E_{1}^2+E_{2}^2}\qquad \abs{E_h}:=\sqrt{E_{1h}^2+E_{2h}^2}. 
\end{equation}
Figures \ref{fig:F2}-\ref{fig:F5} present convergence of P1
finite element  numerical
scheme for the function $\varepsilon$ defined by \eqref{eps}   for 
$m=2,...,9$, see Figure \ref{fig:F1}
for this function.  Tables 1-4 present convergence  rates $q_1, q_2$   for $m=2,5,7,9$ which are computed accordingly
\begin{equation}
  \begin{split}
  q_1 &= \frac{\log \left( \frac{  el^1_h }{  el^2_{2h}} \right)}{\log(0.5)},\\
  q_2 &= \frac{\log \left( \frac{  el^2_h}{ el^2_{2h}} \right)}{\log(0.5)},
  \end{split}
  \end{equation}
where $el^{1,2}_h, el^{1,2}_{2h} $ are computed relative norms $el^{1,2} $ on the 
mesh ${\mathcal T}_h$ with the mesh size $h$ and $2h$, respectively.
Similar convergence rates are obtained for $m=3,4,5,8$.
Using  these tables and figures  we observe that our scheme behaves like a 
first order method  for $H^1(\Omega)$-norm
and second order method for $L^2(\Omega)$-norm  for all values of $m$. 
These results are all in good agreement with the analytic 
estimates derived in Theorems 1-3.

\begin{figure}
\begin{center}
\begin{tabular}{cccc}
  {\includegraphics[scale=0.2, clip=]{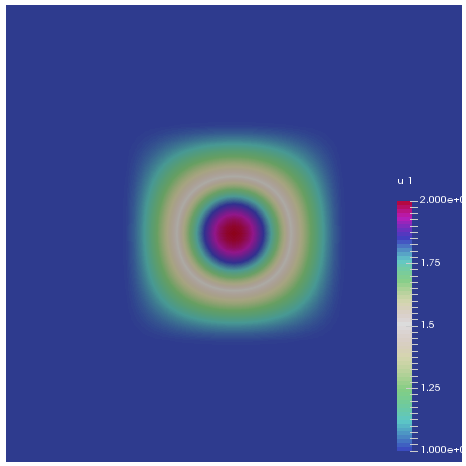}} &
  {\includegraphics[scale=0.2, clip=]{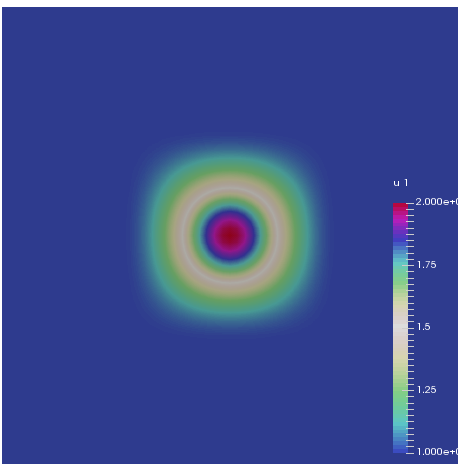}} &
   {\includegraphics[scale=0.2, clip=]{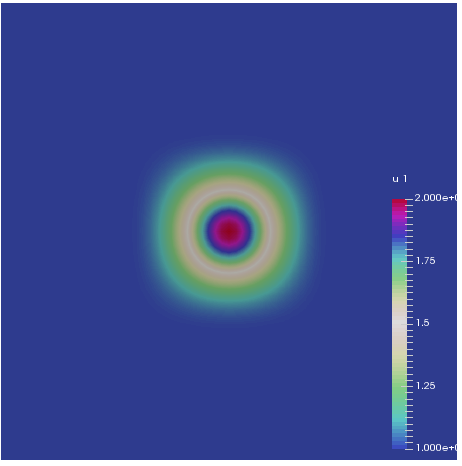}} &
  {\includegraphics[scale=0.2, clip=]{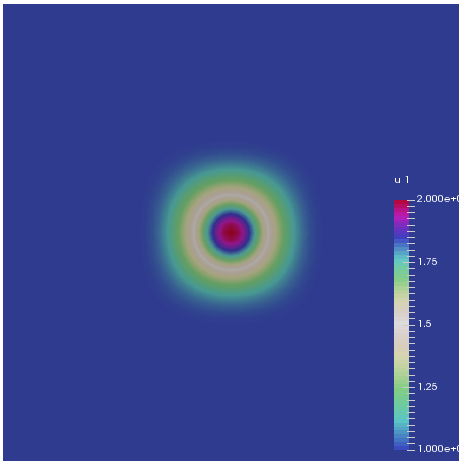}} \\
  $m=2$ & $m=3$ & $m=4$ & $m=5$  \\
    {\includegraphics[scale=0.2, clip=]{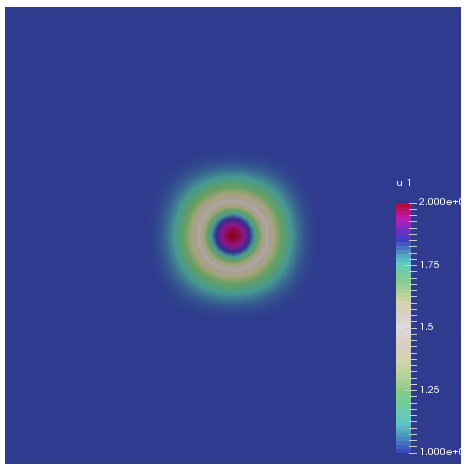}} &
  {\includegraphics[scale=0.2, clip=]{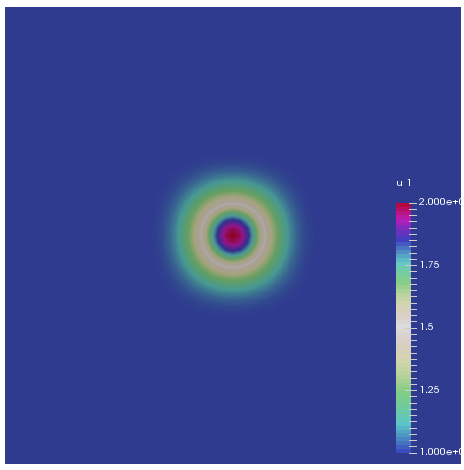}} &
   {\includegraphics[scale=0.2, clip=]{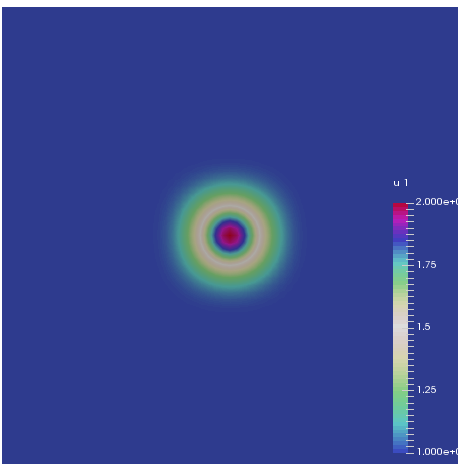}} &
  {\includegraphics[scale=0.2, clip=]{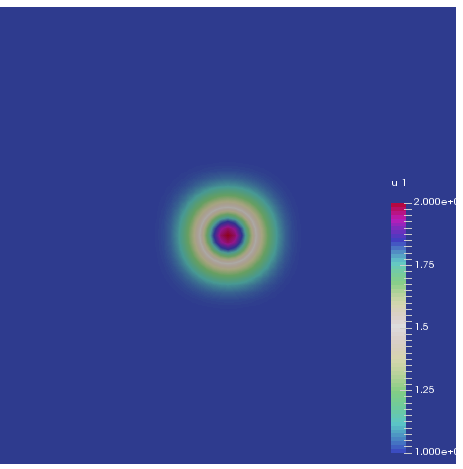}} \\
  $m=6$ & $m=7$ & $m=8$ & $m=9$  \\
\end{tabular}
\end{center}
\caption{{\protect\tiny \emph{The function $\varepsilon(x,y)$ in the domain $\Omega=[0,1] \times[0,1]$  for   different $m$ in \eqref{eps}. }}}
\label{fig:F1}
\end{figure}

\begin{table}[h!] 
\center
\begin{tabular}{ l  l l l l l l  }
\hline
$l$ &  $nel$  & $nno$ &  $e_l^1$  & $q_1$ & $e_l^2$  & $q_2$ \\
\hline 
$1$ & $8$ & $9$       &  $2.71\cdot 10^{-2}$  &              & $8.60\cdot 10^{-2}$ &               \\
$2$ & $32$ & $25$      &  $6.66\cdot 10^{-3}$   &  $2.02$   & $3.25\cdot 10^{-2}$  &  $1.40$  \\
$3$ & $128$ & $81$    &  $1.78\cdot 10^{-3}$  &  $1.90$  &  $1.75\cdot 10^{-2}$  &  $8.99\cdot 10^{-1}$  \\
$4$ & $512$ & $289$   &  $4.13\cdot 10^{-4}$ &   $2.11$  & $1.02\cdot 10^{-2}$  &  $7.79\cdot 10^{-1}$  \\
$5$ & $2048$ & $1089$ &  $1.05\cdot 10^{-4}$   & $1.97$    &  $5.29\cdot 10^{-3}$   & $9.42\cdot 10^{-1} $  \\
$6$ & $8192$ & $4225$ &  $2.65\cdot 10^{-5}$  &   $1.99$  & $2.70\cdot 10^{-3}$  &  $9.69\cdot 10^{-1}  $  \\
\hline
\end{tabular}
\caption{Relative errors in the $L_2$-norm and in the $H^1$-norm for
  mesh sizes $h_l= 2^{-l}, l=1,...,6$,  for $m=2$ in \eqref{eps}.}
\label{testm2}
\end{table}

\begin{table}[h!] 
\center
\begin{tabular}{ l  l l l l l l  }
\hline
$l$ &  $nel$  & $nno$ &  $e_l^1$  & $q_1$ & $e_l^2$  & $q_2$ \\
\hline 
$1$ & $8$ & $9$       &  $2.35\cdot 10^{-2}$  &              & $1.20\cdot 10^{-1}$ &               \\
$2$ & $32$ & $25$      &  $5.02\cdot 10^{-3}$   &  $2.22$   & $5.18\cdot 10^{-2}$  &  $1.21$  \\
$3$ & $128$ & $81$    &  $1.24\cdot 10^{-3}$  &  $2.02$  &  $2.69\cdot 10^{-2}$  &  $9.45\cdot 10^{-1}$  \\
$4$ & $512$ & $289$   &  $2.95\cdot 10^{-4}$ &   $2.07$  & $1.06\cdot 10^{-2}$  &  $1.34$  \\
$5$ & $2048$ & $1089$ &  $7.67\cdot 10^{-5}$   & $1.94$    &  $5.40\cdot 10^{-3}$   & $9.74\cdot 10^{-1} $  \\
$6$ & $8192$ & $4225$ &  $1.94\cdot 10^{-5}$  &   $1.99$  & $2.72\cdot 10^{-3}$  &  $9.92\cdot 10^{-1}  $  \\
\hline
\end{tabular}
\caption{Relative errors in the $L_2$-norm and in the $H^1$-norm for
  mesh sizes $h_l= 2^{-l}, l=1,...,6$,  for $m=5$ in \eqref{eps}.}
\label{testm5}
\end{table}

\begin{table}[h!] 
\center
\begin{tabular}{ l  l l l l l l  }
\hline
$l$ &  $nel$  & $nno$ &  $e_l^1$  & $q_1$ & $e_l^2$  & $q_2$ \\
\hline 
$1$ & $8$ & $9$       &  $2.28\cdot 10^{-2}$  &              & $1.15\cdot 10^{-1}$ &               \\
$2$ & $32$ & $25$      &  $4.45\cdot 10^{-3}$   &  $2.36$   & $4.47\cdot 10^{-2}$  &  $1.35$  \\
$3$ & $128$ & $81$    &  $1.09\cdot 10^{-3}$  &  $2.03$  &  $2.41\cdot 10^{-2}$  &  $8.92\cdot 10^{-1}$  \\
$4$ & $512$ & $289$   &  $2.62\cdot 10^{-4}$ &   $2.05$  & $1.08\cdot 10^{-2}$  &  $1.16$  \\
$5$ & $2048$ & $1089$ &  $6.76\cdot 10^{-5}$   & $1.95$    &  $5.32\cdot 10^{-3}$   & $1.01 $  \\
$6$ & $8192$ & $4225$ &  $1.71\cdot 10^{-5}$  &   $1.98$  & $2.66\cdot 10^{-3}$  &  $9.98\cdot 10^{-1}  $  \\
\hline
\end{tabular}
\caption{Relative errors in the $L_2$-norm and in the $H^1$-norm for
  mesh sizes $h_l= 2^{-l}, l=1,...,6$, for $m=7$ in \eqref{eps}.}
\label{testm7}
\end{table}

\begin{table}[h!] 
\center
\begin{tabular}{ l  l l l l l l  }
\hline
$l$ &  $nel$  & $nno$ &  $e_l^1$  & $q_1$ & $e_l^2$  & $q_2$ \\
\hline 
$1$ & $8$ & $9$       &  $1.73\cdot 10^{-2}$  &              & $7.29\cdot 10^{-2}$ &               \\
$2$ & $32$ & $25$      &  $3.33\cdot 10^{-3}$   &  $2.38$   & $3.57\cdot 10^{-2}$  &  $1.03$  \\
$3$ & $128$ & $81$    &  $8.98\cdot 10^{-4}$  &  $1.89$  &  $2.15\cdot 10^{-2}$  &  $7.33\cdot 10^{-1}$  \\
$4$ & $512$ & $289$   &  $2.36\cdot 10^{-4}$ &   $1.93$  & $1.08\cdot 10^{-2}$  &  $9.94\cdot 10^{-1}$  \\
$5$ & $2048$ & $1089$ &  $6.09\cdot 10^{-5}$   & $1.96$    &  $5.26\cdot 10^{-3}$   & $1.04 $  \\
$6$ & $8192$ & $4225$ &  $1.55\cdot 10^{-5}$  &   $1.98$  & $2.62\cdot 10^{-3}$  &  $1.00 $  \\
\hline
\end{tabular}
\caption{Relative errors in the $L_2$-norm and in the $H^1$-norm for
  mesh sizes $h_l= 2^{-l}, l=1,...,6$,  for $m=9$ in \eqref{eps}.}
\label{testm9}
\end{table}

\begin{figure}[h!]
\begin{center}
\begin{tabular}{cc}
  {\includegraphics[scale=0.5, clip=]{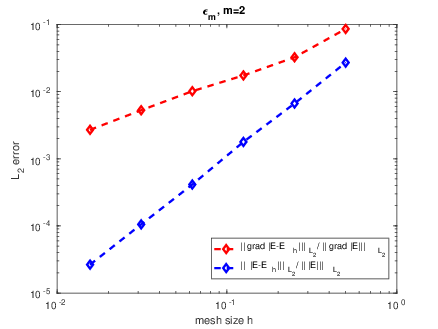}} &
  {\includegraphics[scale=0.5, clip=]{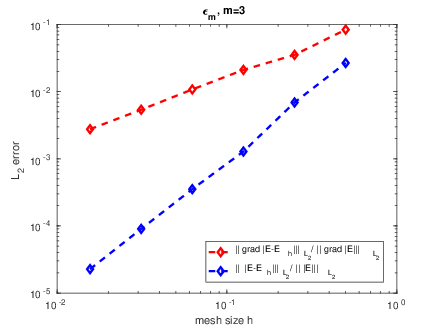}}   
\end{tabular}\
\end{center}
\caption{Relative errors for $m=2$ (left) and $m=3$ (right).}
\label{fig:F2}
\end{figure}

\begin{figure}[h!]
\begin{center}
\begin{tabular}{cc}
  {\includegraphics[scale=0.5, clip=]{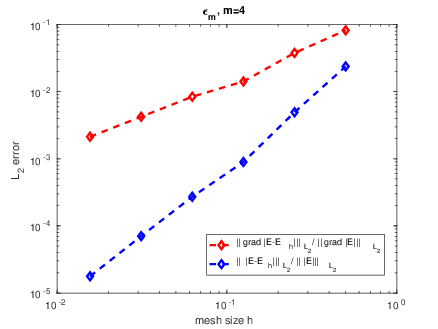}} &
  {\includegraphics[scale=0.5, clip=]{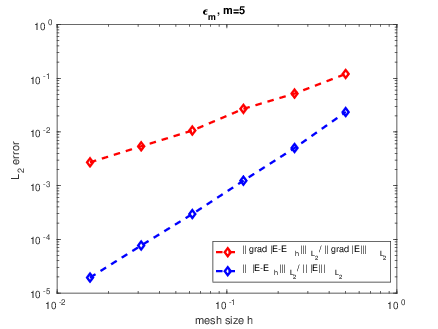}}   
\end{tabular}\
\end{center}
\caption{Relative errors for $m=4$ (left) and $m=5$ (right).}
\label{fig:F3}
\end{figure}

\begin{figure}[h!]
\begin{center}
\begin{tabular}{cc}
  {\includegraphics[scale=0.5, clip=]{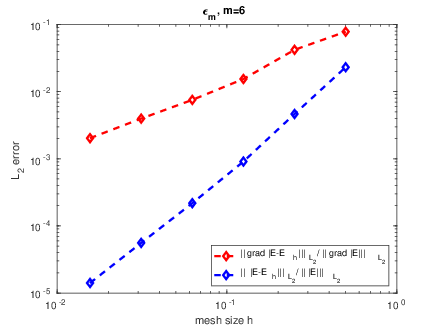}} &
  {\includegraphics[scale=0.5, clip=]{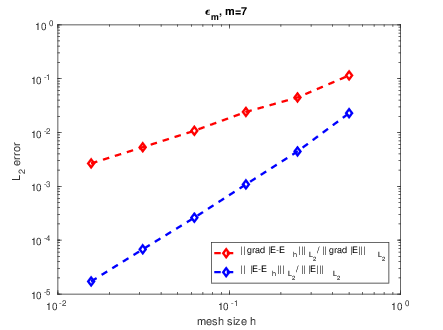}}   
\end{tabular}\
\end{center}
\caption{Relative errors for $m=6$ (left) and $m=7$ (right).}
\label{fig:F4}
\end{figure}

\begin{figure}[h!]
\begin{center}
\begin{tabular}{cc}
  {\includegraphics[scale=0.5, clip=]{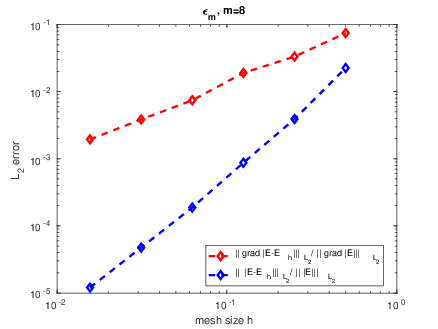}} &
  {\includegraphics[scale=0.5, clip=]{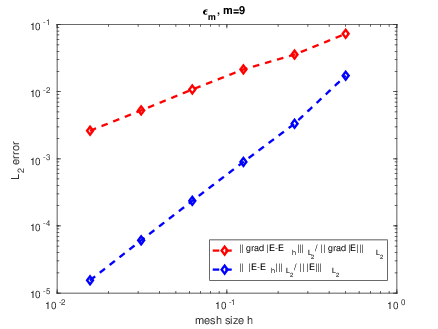}}   
\end{tabular}\
\end{center}
\caption{Relative errors for $m=8$ (left) and $m=9$ (right).}
\label{fig:F5}
\end{figure}

\begin{figure}[h!]
\begin{center}
  \begin{tabular}{cccc}
    \hline
    \multicolumn{4}{c}{ $|E_h|, m=2$}\\
  {\includegraphics[scale=0.2, clip=]{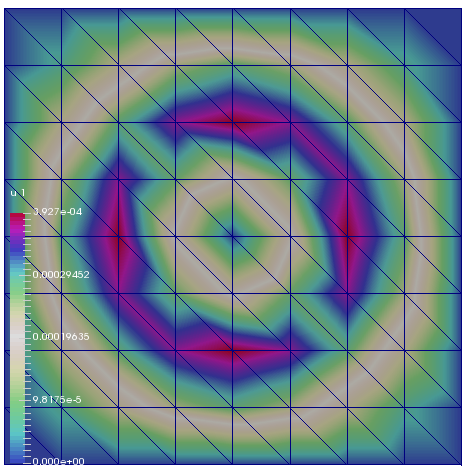}}  &
  {\includegraphics[scale=0.2, clip=]{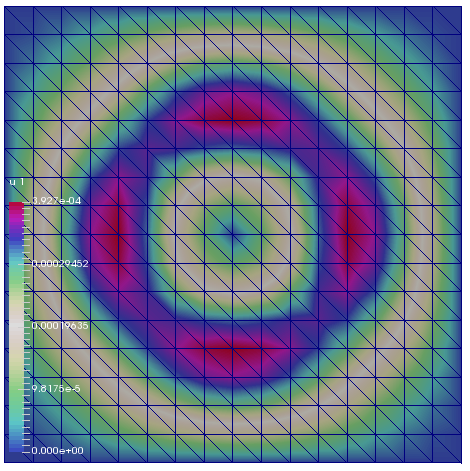}} &
  {\includegraphics[scale=0.2, clip=]{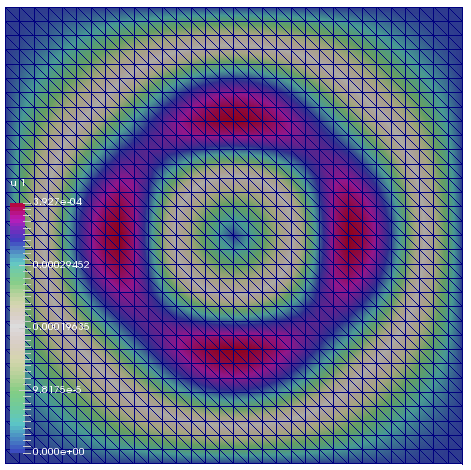}}  &
   {\includegraphics[scale=0.2, clip=]{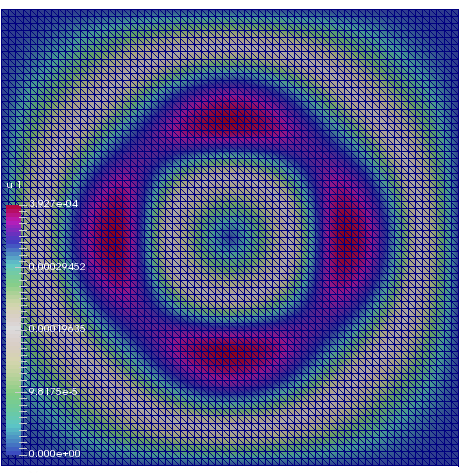}}\\
    $h=0.125$ &  $h=0.0625$ & $h=0.03125$ & $h=0.015625$\\
   \hline
     \multicolumn{4}{c}{ $| E|, m=2$}\\
     {\includegraphics[scale=0.2, clip=]{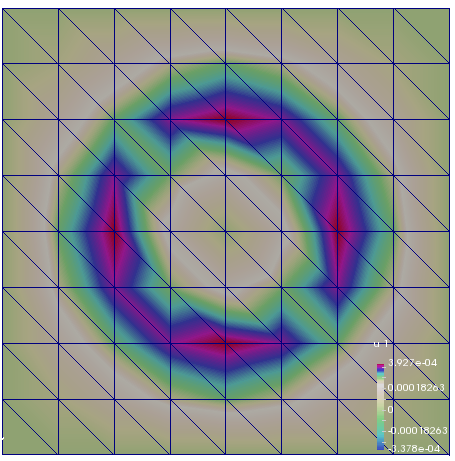}}  &
  {\includegraphics[scale=0.2, clip=]{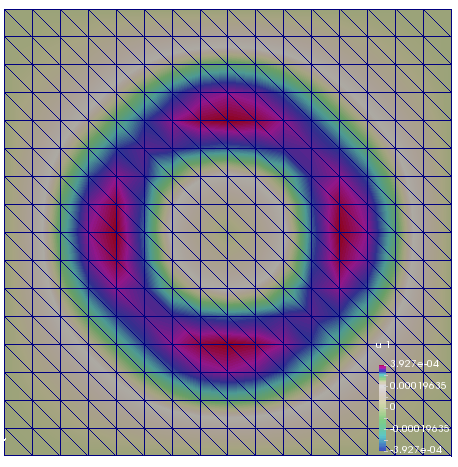}} &
  {\includegraphics[scale=0.2, clip=]{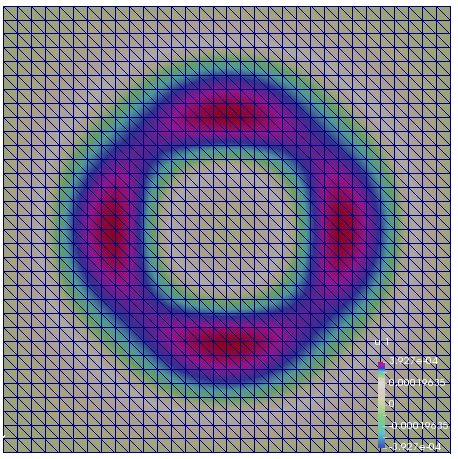}}  &
   {\includegraphics[scale=0.2, clip=]{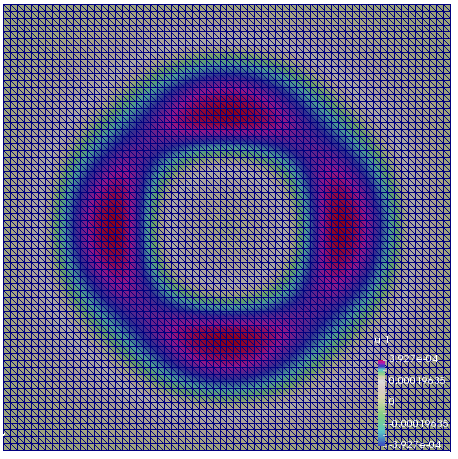}}\\
     $h=0.125$ &  $h=0.0625$ & $h=0.03125$ & $h=0.015625$\\
   \hline
      \multicolumn{4}{c}{ $|E_h|, m=5$}\\
  {\includegraphics[scale=0.2, clip=]{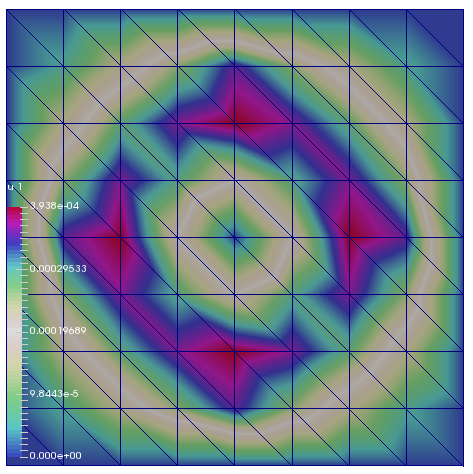}}  &
  {\includegraphics[scale=0.2, clip=]{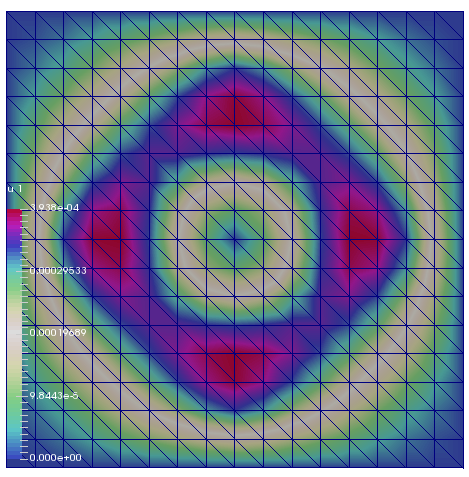}} &
  {\includegraphics[scale=0.2, clip=]{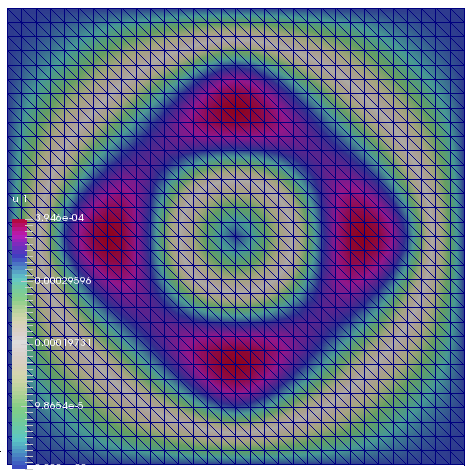}}  &
   {\includegraphics[scale=0.2, clip=]{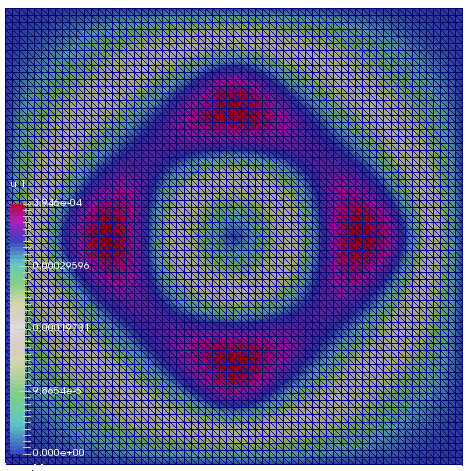}}\\
    $h=0.125$ &  $h=0.0625$ & $h=0.03125$ & $h=0.015625$\\
   \hline
     \multicolumn{4}{c}{ $|E|, m=5$}\\
     {\includegraphics[scale=0.2, clip=]{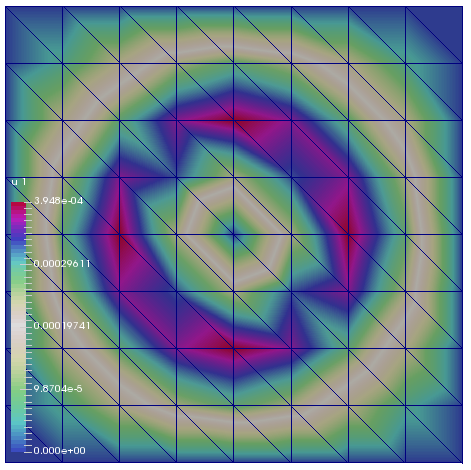}}  &
  {\includegraphics[scale=0.2, clip=]{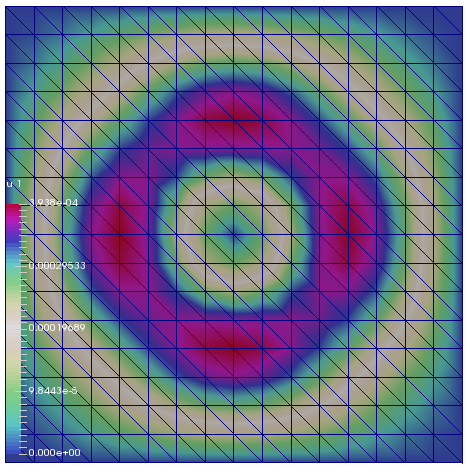}} &
  {\includegraphics[scale=0.2, clip=]{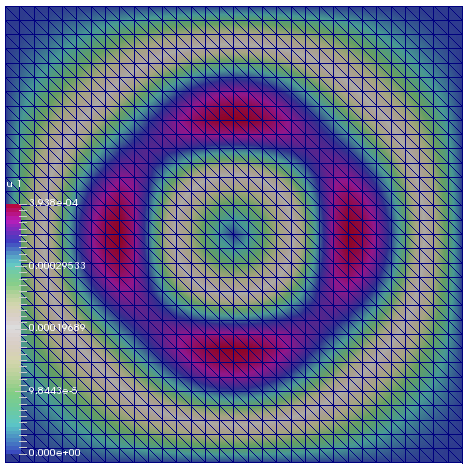}}  &
   {\includegraphics[scale=0.2, clip=]{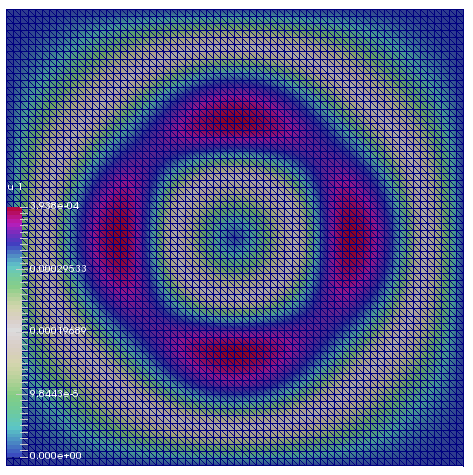}}\\
     $h=0.125$ &  $h=0.0625$ & $h=0.03125$ & $h=0.015625$\\ 
   \hline
\end{tabular}
\end{center}
\caption{
    \emph{Computed vs. exact solution for different meshes taking $m=2$  and $m=5$  in \eqref{eps}.}}
\label{fig:F6}
\end{figure}

\begin{figure}[h!]
\begin{center}
  \begin{tabular}{cccc}
    \hline
    \multicolumn{4}{c}{ $|E_h|, m=7$}\\
  {\includegraphics[scale=0.2, clip=]{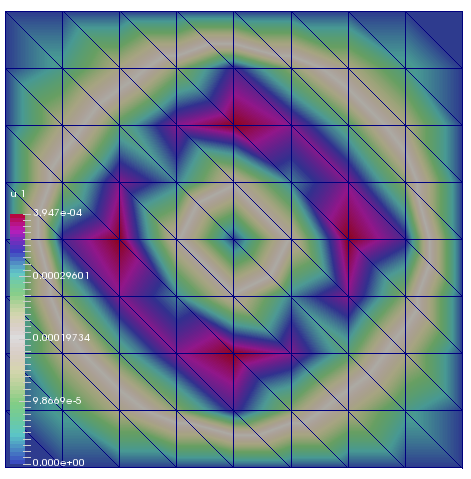}}  &
  {\includegraphics[scale=0.2, clip=]{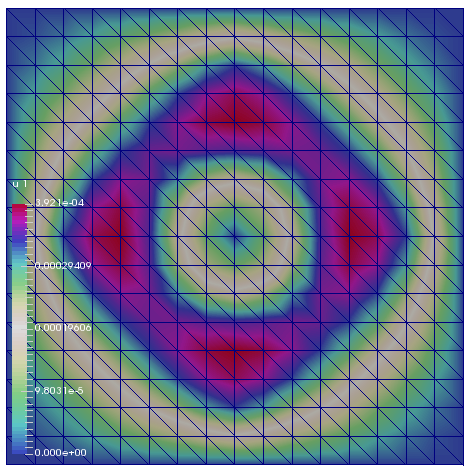}} &
  {\includegraphics[scale=0.2, clip=]{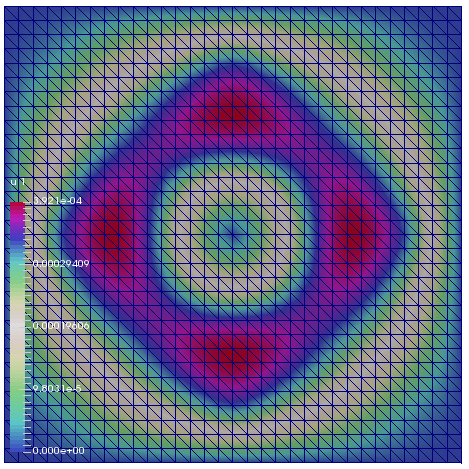}}  &
   {\includegraphics[scale=0.2, clip=]{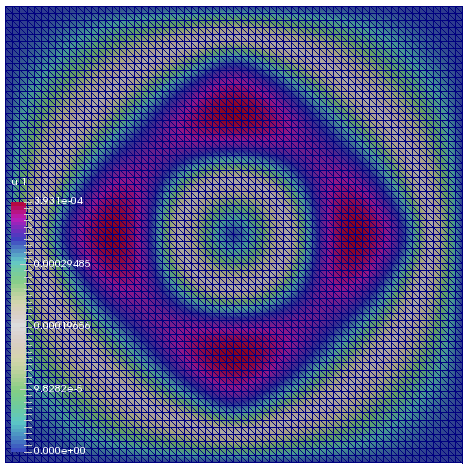}}\\
    $h=0.125$ &  $h=0.0625$ & $h=0.03125$ & $h=0.015625$\\
   \hline
     \multicolumn{4}{c}{ $|E|, m=7$}\\
      {\includegraphics[scale=0.2, clip=]{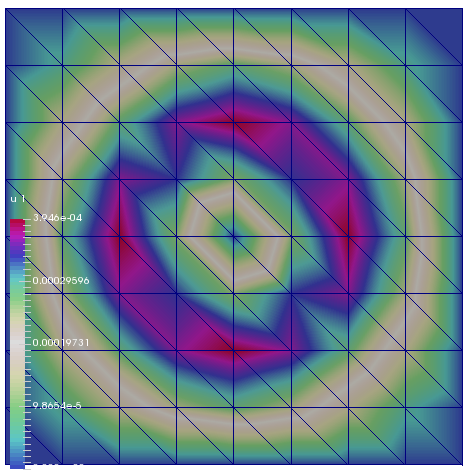}}  &
  {\includegraphics[scale=0.2, clip=]{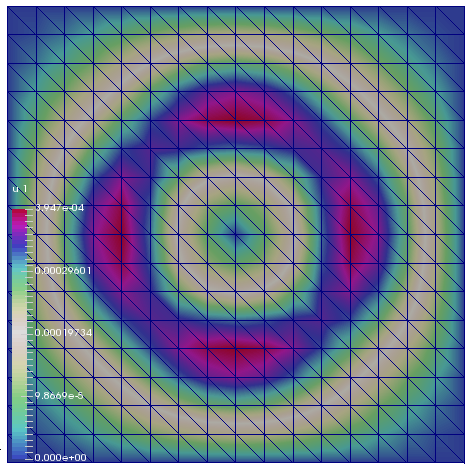}} &
  {\includegraphics[scale=0.2, clip=]{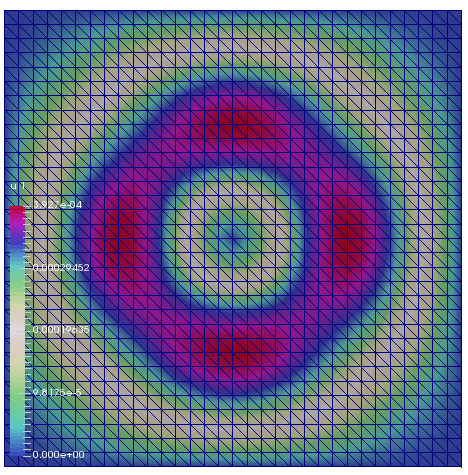}}  &
   {\includegraphics[scale=0.2, clip=]{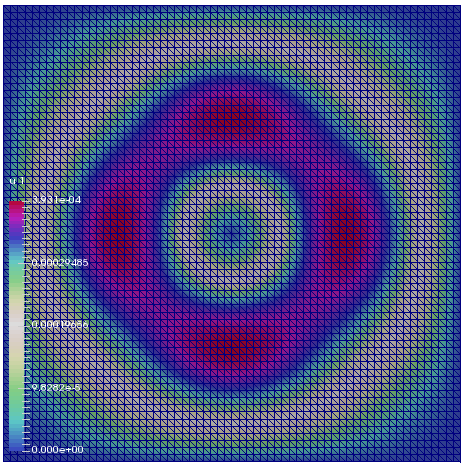}}\\
     $h=0.125$ &  $h=0.0625$ & $h=0.03125$ & $h=0.015625$\\
   \hline
      \multicolumn{4}{c}{ $|E_h|, m=9$}\\
  {\includegraphics[scale=0.2, clip=]{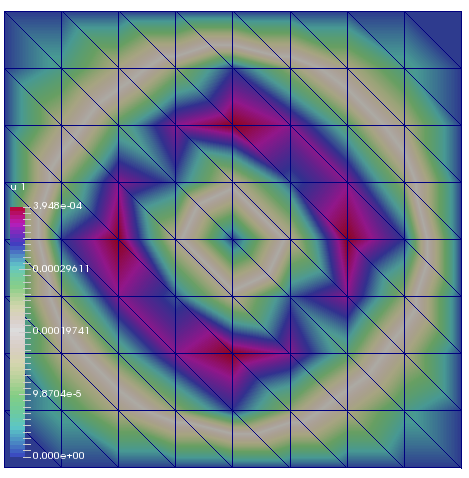}}  &
  {\includegraphics[scale=0.2, clip=]{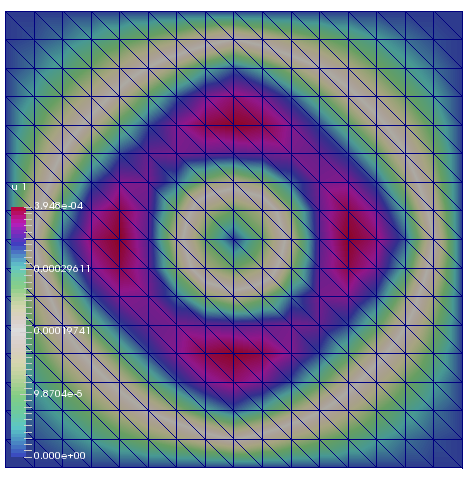}} &
  {\includegraphics[scale=0.2, clip=]{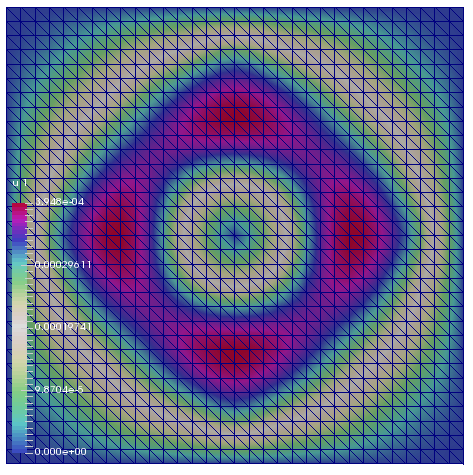}}  &
   {\includegraphics[scale=0.2, clip=]{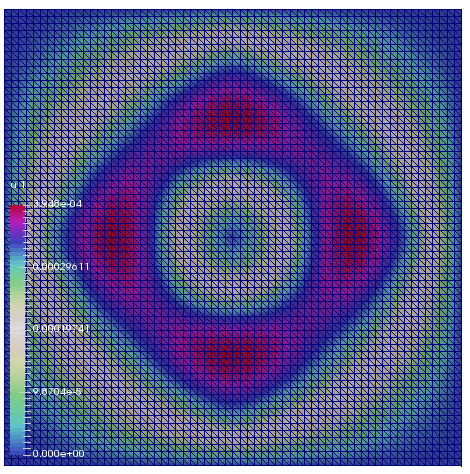}}\\
    $h=0.125$ &  $h=0.0625$ & $h=0.03125$ & $h=0.015625$\\
   \hline
     \multicolumn{4}{c}{ $|E|, m=9$}\\
     {\includegraphics[scale=0.2, clip=]{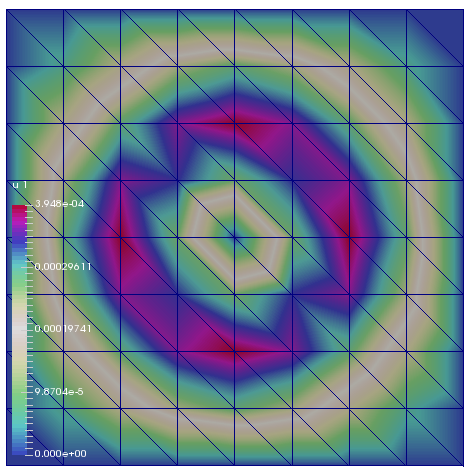}}  &
  {\includegraphics[scale=0.2, clip=]{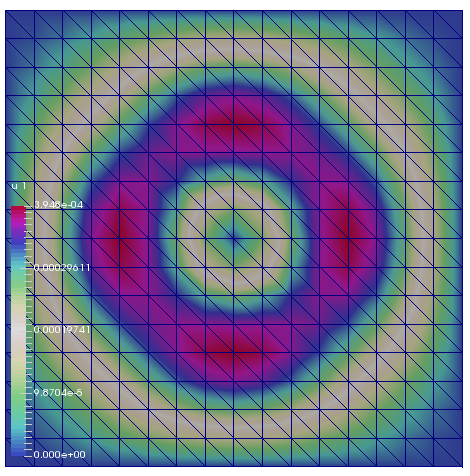}} &
  {\includegraphics[scale=0.2, clip=]{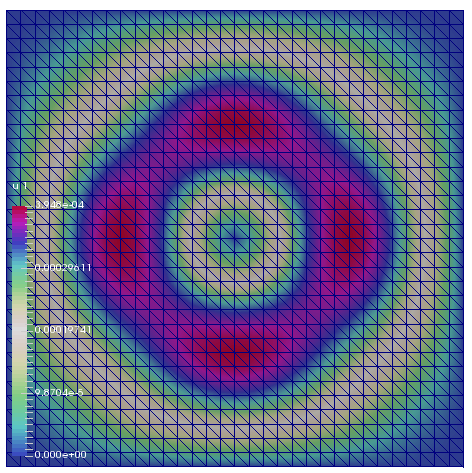}}  &
   {\includegraphics[scale=0.2, clip=]{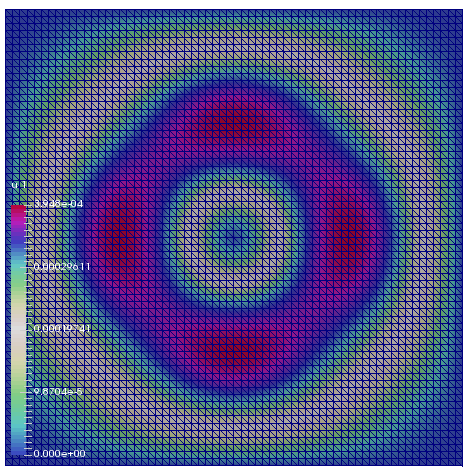}}\\
     $h=0.125$ &  $h=0.0625$ & $h=0.03125$ & $h=0.015625$\\
   \hline
\end{tabular}
\end{center}
\caption{
    \emph{Computed vs. exact solution for different meshes taking $m=2$  and $m=5$  in \eqref{eps}.}}
\label{fig:F7}
\end{figure}

\section{Conclusion} 
In this paper we consider the time harmonic 
Maxwell's equations obtained through Laplace transform applied 
to a time-dependent model problem with a certain variable, positive, dielectric 
permittivity function $\varepsilon (x)$.  Due to the varying nature of 
the dielectric 
permittivity function, we study the problem setting in a rectangular 
(cubic in $3d$) domain 
$\Omega$ split into an axi-parallel interior subdomain $\Omega_1$ 
with varying $\varepsilon(x)$ and an outer domain 
$\Omega_2:=\Omega\setminus\Omega_1$, where 
$\varepsilon\equiv 0$.  Thus, $\Omega_1$ and $\Omega_2$ 
are two disjoint open sets with a, partially, common boundary (see Fig. 1). 
In multiscale setting they may be considered as 
the domains of fine and coarse numerical resolutions. 
We construct a $P1$ stabilized finite element 
approximations for this problem and prove its consistency and well-posedness. 
As for the accuracy of the constructed numerical scheme we  
derive, optimal,  
a priori and a posteriori error bounds, 
in some, gradient dependent, weighted energy norms. 

The optimality is conformed both in forward and dual constructions 
which yield the same convergence  rates. 
The involved weight function appears in 
different forms and combines the dielectric permittivity and the 
Laplace transformation variable. We have  
implemented several numerical 
examples that validate the robustness of the theoretical studies. 

This problem is of vital importance in solving the 
{\sl Coefficient Inverse Problems} (CIPs) with several applications ranging 
from medical physics (radiation therapy) to micro turbines, computer chips, 
devices in fusion energy studies and so on. 
For the study of a corresponding time-dependent 
problem we refer to, e.g., \cite{BR}. 
In many application, to determine a reliable dielectric 
permittivity function, the
 time-dependent Maxwell's equations are required 
in the entire space $\mathbb{R}^{3}$, see, e.g.,  \cite{BK, BTKM1,
  BTKM2, BondestaB, TBKF1, TBKF2}. In present study, however,  
 it is more efficient to consider
Maxwell's equations in bounded domain, 
with constant dielectric permittivity function in
a neighborhood of the boundary of the computational domain. 
The  P1 finite element scheme of \cite{BMaxwell} is used
for solution of different CIPs to determine the dielectric
permittivity function in non-conductive media for time-dependent
Maxwell's equations using simulated and experimentally generated data,
see, e.g., 
\cite{BMedical, BeilinaHyb, BK,  BTKM1, BTKM2, BondestaB, TBKF1, TBKF2}. 


\vspace{0.5cm}

\noindent \textbf{\underline{Acknowledgment:}} The research of both authors
is supported by the Swedish Research Council grant VR 2018-03661. The first 
author acknowledges the support of the VR grant DREAM. 
\rule{2mm}{2mm}


\end{document}